\documentclass[a4paper, 12pt]{article}
\usepackage [a4paper,left=2.5cm,bottom=2.5cm,right=2.5cm,top=2.5cm]{geometry}
\usepackage[english]{babel}
\usepackage{amssymb}
\usepackage{amsmath,amsthm, dsfont}
\usepackage{subcaption}
\usepackage{comment}
\usepackage{amsmath}
\usepackage{tabularx,array}
\usepackage{graphicx, url}
\usepackage{enumitem} 
\usepackage{bm}
\usepackage{float}
\usepackage{bbm}
\usepackage{xcolor}
\usepackage{hyperref}
\usepackage{graphicx}
\usepackage{todonotes}
\usepackage{todonotes}
\usepackage{hyperref}
\hypersetup{colorlinks=true, linkcolor=blue,citecolor=gray}
\usepackage[nameinlink,capitalise]{cleveref}
\crefformat{equation}{#2(#1)#3}

\usepackage{color,soul}

\theoremstyle{plain}
\theoremstyle{plain}\newtheorem{assumption}{Assumption}
\crefname{assumption}{Assumption}{Assumptions}
\newtheorem{theorem}{Theorem}[section]
\newtheorem{lemma}[theorem]{Lemma}
\newtheorem{corollary}[theorem]{Corollary}
\newtheorem{proposition}[theorem]{Proposition}

\newtheorem{remark}[theorem]{Remark}

\newcommand{\argmin}[1]{\underset{#1}
{\operatorname{arg}\!\operatorname{min}}\;}
\newcommand{\one}{\mathds{1}}

\newcommand{\E}{\mathbb{E}}

\newcommand{\R}{\mathbb{R}}
\newcommand{\Z}{\mathbb{Z}}
\newcommand{\C}{\mathbb{C}}
\newcommand{\N}{\mathbb{N}}

\renewcommand{\P}{\mathbb{P}}

\newcommand{\dn}{\Delta_n}
\newcommand{\nnatural}[1]{\left\| #1 \right\|_{N,n}}

\newcommand{\nstarn}[1]{\left\| #1 \right\|_{\star,n}}

\newcommand{\sumsum}{\sum_{i=1}^N \sum_{j=0}^{n-1}}

\newcommand{\pnatural}[1]{\left\langle #1 \right\rangle_{N,n}}

\newcommand{\pstarn}[1]{\left\langle #1 \right\rangle_{\star,n}}

\title{Nonparametric Estimation of  the Diffusion-Interaction Function in  Particle Systems}
\author{Chiara Amorino \thanks{Universitat Pompeu Fabra and Barcelona School of Economics, Department of Economics and Business,
Ramon Trias Fargas 25-27, 08005, Barcelona, Spain
Email: chiara.amorino@upf.edu}\and  Francisco Pina \thanks{Department of Mathematics, University of Luxembourg, Maison du Nombre, 6 Avenue de la Fonte, 4364 Esch-sur-Alzette,
Luxembourg, Email: francisco.pina@uni.lu} \and 
Mark Podolskij\thanks{Department of Mathematics, University of Luxembourg, Maison du Nombre, 6 Avenue de la Fonte, 4364 Esch-sur-Alzette,
Luxembourg, Email: mark.podolskij@uni.lu}
}

\usepackage{autonum}
\numberwithin{equation}{section}

\date{\today}
\begin{document}
\maketitle

\begin{abstract}
\noindent
We study the problem of nonparametric estimation of the interaction function in   particle systems whose interaction enters through the diffusion coefficient. The particle system is observed discretely over a fixed time interval, and the interaction function is assumed to belong to a suitable nonparametric function class. We propose an estimation procedure based on empirical risk minimization over finite-dimensional sieve spaces, introducing a computable contrast function that exploits the local volatility of the observed trajectories. We derive a general non-asymptotic error bound under the intrinsic norm  associated with the statistical problem and, focusing the analysis on multiresolution wavelet approximation spaces, we obtain explicit polynomial convergence rates in two distinct regimes, depending on whether the particle or discretization error dominates. Finally, we establish  convergence rates for the classical $L^2$-norm.  \\

 

\end{abstract}

\tableofcontents
\section{Introduction}
Stochastic differential equations (SDEs) constitute a fundamental mathematical framework for modelling dynamical systems subject to random fluctuations and have become an indispensable tool in probability theory and its applications \cite{kar91,oks03}. While classical SDEs describe the evolution of isolated stochastic systems, many phenomena of practical interest involve large populations of interacting agents whose dynamics are influenced by the collective behaviour of the system. Such models naturally lead to systems of interacting stochastic differential equations, commonly referred to as interacting particle systems (IPS).
Originating from the pioneering works of Kac and McKean \cite{kac56,mck66}, IPS have become one of the central objects of modern probability theory. Over the past decades, their mathematical analysis has led to a rich theoretical framework, including fundamental results on propagation of chaos, McKean--Vlasov limits, well-posedness and stability; see, for instance, \cite{car18,cha22,mel96,szn91}. Beyond their mathematical interest, IPS have become a fundamental modelling paradigm across a wide range of scientific disciplines. They naturally describe collective phenomena arising from the interaction of many individual agents and have found applications in finance \cite{car19, dje22,gie20}, statistical physics \cite{mar13,dem86,spo91}, biology \cite{mun14}, neuroscience \cite{bal12, tou14} and  social sciences \cite{alb19, gom12, nal10}. More recently, IPS and mean-field models have also emerged as an important mathematical framework in modern machine learning, particularly in the analysis of neural networks \cite{deb20, chi18,rot22}, Sequential Monte Carlo methods \cite{del13}, diffusion generative models \cite{par24} and transformer-based language models \cite{ges25,rig25}, where interacting particle systems provide a natural mathematical framework for analysing the collective dynamics induced by self-attention mechanisms.

The numerous applications of interacting particle systems naturally motivate the development of a solid statistical framework. Although the probabilistic theory has been extensively developed over the past decades, the statistical analysis of IPS has become an active area of research only in recent years. Most of the existing works concern parametric and semiparametric inference, where the interaction mechanism depends on a finite-dimensional parameter to be estimated from the observations of the particle system. Under various observation schemes and asymptotic regimes, several estimation procedures have been proposed, together with consistency, asymptotic normality and efficiency results \cite{amo23, amorino2025polynomial, bel23, bis11,che21,del23,gen21,gen21.02,liu22,sha23} among others. 

In contrast, the literature on nonparametric inference for IPS is still rather limited. Existing nonparametric approaches include kernel methods for drift estimation \cite{della22}, moment-based methods for estimating the interaction function \cite{com25}, nonparametric maximum likelihood estimation of the drift field \cite{yao22}, and empirical risk minimization over finite-dimensional sieve spaces \cite{bel24}. Despite relying on different estimation techniques, all these contributions are concerned with models in which the interaction enters the drift coefficient while the diffusion coefficient does not depend on the empirical measure of the system. Models with distribution-dependent diffusion coefficients, however, naturally arise in many applications and provide a rich modelling framework. To the best of our knowledge, nonparametric inference for such models has not yet been investigated.

Motivated by this gap, we study the problem of estimating the unknown interaction function from discrete observations of an interacting particle system  described by the following SDE: 
\begin{equation}\label{eqIPS}
dX_t^i = b(X_t^i, \mu_t^N) dt + \left( \varphi \star \mu_t^N\right)(X_t^i)dW_t^i, \hspace{0.3cm} i=1,\dots,N,
\end{equation}
where $t\in [0,T]$ for a fixed time horizon $T$, $b:\R\times\mathcal{P}_2({\R})\longrightarrow {\R}$ denotes the drift function, $\varphi:\mathbb{R}\rightarrow\mathbb{R}$ is the unknown interaction function and $\{W_t^i\}_{i=1}^N$ is a collection of independent Brownian motions. Furthermore, $\mu_t^N$ denotes the empirical measure of the particle system at time $t$, defined by
\begin{equation}\label{eq:empiricalMeasureIPS}
\mu_t^N=\frac{1}{N}\sum_{i=1}^N\delta_{X_t^i},
\end{equation}
while $\varphi\star\mu_t^N$
denotes the convolution between the interaction function and the empirical measure.
A central role in our analysis is played by the mean-field limit of the interacting particle system. Indeed, under suitable assumptions on the coefficients, the propagation of chaos property implies that, as $N \rightarrow \infty$, the empirical measure $\mu_t^N$ of \eqref{eqIPS} converges weakly to $\mu_t$, the law of a limiting process whose dynamics are governed by the McKean--Vlasov stochastic differential equation
\begin{equation}\label{eqMcKeanVlasovSDE}
\left\{
\begin{aligned}
&d\overline{X}_t
= b(\overline{X}_t, \mu_t)\,dt
+ \bigl(\varphi \star \mu_t\bigr)(\overline{X}_t)\,dB_t, \\
&\mathrm{Law}(\overline{X}_t) = \mu_t,
\end{aligned}
\right.
\end{equation}
where $B$ is a Brownian motion.

The objective of this paper is to estimate the unknown interaction function $\varphi$ from discrete observations of the particle system \eqref{eqIPS}. Our approach is based on empirical risk minimization over a sequence of finite-dimensional sieve spaces. If on one side it follows the general philosophy developed in \cite{bel24} for the estimation of the interaction function in the drift, on the other moving the interaction from the drift to the diffusion fundamentally changes the statistical inference problem. First, since the interaction function enters the diffusion coefficient, the empirical contrast used in \cite{bel24} is no longer suitable and a new  contrast based on the local volatility of the observed trajectories is introduced in Section \ref{sec:MinimizationCompactFE}. Second, the theoretical analysis requires a fundamentally different approach. Indeed, a common strategy in the existing literature  for drift-interaction particle systems, adopted in particular in \cite{bel24}, is to first analyze the statistical problem in the McKean--Vlasov setting, where the particles are independent, and then transfer the resulting estimates to the interacting particle system through the Girsanov-based change-of-measure technique  introduced in \cite{della22}. However, this procedure is no longer available in the present setting, since it fundamentally relies on the coincidence of the diffusion coefficients of the interacting system and its McKean--Vlasov limit, which no longer holds when the interaction enters the diffusion coefficient.  Instead, our analysis is performed directly on the interacting  system, relying on quantitative propagation-of-chaos estimates in Wasserstein distance. Finally, unlike the continuous-observation framework considered in \cite{bel24}, the present work deals with a discrete observation scheme, so that the impact of the discretization error must also be quantified and carefully incorporated into the statistical analysis. {Indeed, the presence of the discretization error leads to two different regimes for the convergence of the estimator, which are carefully explained in the article and differ from previous work.}

In the present paper, we first introduce a computable contrast function and define an estimator based on empirical risk minimization over finite-dimensional sieve spaces. We then derive a non-asymptotic error bound under the intrinsic norm naturally associated with the statistical model. Next, specializing the analysis to wavelet approximation spaces, we obtain explicit bounds for the approximation error and derive convergence rates under the intrinsic loss. Finally, we investigate the performance of the estimator under the classical $L^2$ loss and establish convergence rates for its estimation error over a suitable class of functions.

The remainder of the paper is organized as follows. In Section \ref{sec:MinimizationCompactFE}, we introduce the proposed estimator and derive a general non-asymptotic error bound under the intrinsic norm of the model. Section \ref{sec:ApproximationError} investigates the approximation properties of multiresolution wavelet sieve spaces and establishes corresponding convergence rates. In Section \ref{sec:L2ERROR}, we study the estimator under the classical $L^2$-norm. Section \ref{sec:AuxiliaryProofs} contains most of the proofs and other auxiliary results.

\subsection{Notation and probabilistic tools}
All the probabilistic elements of the paper are defined on a filtered probability space $(\Omega,\mathcal{F},(\mathcal{F}_t)_{t\geq 0},\P).$
For $p \geq 1$ let $\mathcal{P}_p(\mathbb{R})$ denote the space of probability measures on $\R$ with finite $p$-th moment, that is,
\begin{equation}
\mathcal{P}_p(\mathbb{R})
:=
\left\{
\mu \in \mathcal{P}(\mathbb{R})
:
\int_{\mathbb{R}} |x|^p\mu(dx)<\infty
\right\}.
\end{equation}
For a set $A \subset \R$, we denote by $A^{\mathcal{C}}$ its complement. For $1\leq p<\infty$, let $L^p(A)$ denote the usual Lebesgue space on $A$,
equipped with the norm
\begin{equation}
    \|f\|_{L^p\left(A\right)} := \left(\int_A |f(x)|^pdx \right)^{1/p}.
\end{equation}
In particular, for $A = \R$, we denote the spaces by $L^p(\R)$ and we write $\|f\|_{L^p}$ for the associated norm. We denote by
\begin{equation}
    \|f\|_{\infty} := \sup_{x\in\R}|f(x)|, \quad \|f\|_{\text{Lip}}:= \sup_{x\neq y }\frac{|f(x) - f(y)|}{|x-y|},
\end{equation} the supremum and the Lipschitz norm of $f$, respectively. Moreover, we denote by 
\begin{equation}
    \widehat{f}(z) := \int_{\R}\exp(-izx)f(x)dx.
\end{equation} the Fourier transform of $f$. For a function $g:\R \rightarrow \R$ we denote the convolution between $f$ and $g$ as 
\begin{equation}
    (f \star g)(x) := \int_{\R}f(x-y)g(y)dy,
\end{equation}
and $\text{supp}(f)$ denotes the support of $f$.
For $p \geq 1$, given two probability measures $\mu,\nu$ on $\mathbb{R}$ with finite $p$-th moments, we define the $p$-Wasserstein distance by
\begin{equation}
W_p(\mu,\nu)
:=
\left(
\inf_{\pi\in\Pi(\mu,\nu)}
\int_{\mathbb{R}\times\mathbb{R}}
|x-y|^p\,d\pi(x,y)
\right)^{1/p},
\end{equation}
where $\Pi(\mu,\nu)$ denotes the set of all couplings of $\mu$ and $\nu$. 
Throughout the paper, we write
\begin{equation}
a \lesssim b
\end{equation}
if there exists a positive constant $C$, independent of $N$ and $n$, such that $a \leq Cb$, and $a \wedge b:= \text{min}(a,b).$
The set $B_R^{\C}(x)$ denotes the closed disk centered at $x$ with radius $R>0$ in the complex plane, and we simply write $B_R^{\mathbb{C}}$ when $x=0$. 

Next, we recall several classical notions and results from empirical process theory  that will be repeatedly used throughout the paper.
The Orlicz norm of a real-valued random variable $X$ with respect to a
nondecreasing, convex function $\psi:[0,\infty)\to[0,\infty)$ satisfying $\psi(0)=0$ is defined by
\begin{equation}
\|X\|_{\psi} := \inf \left\{ c > 0 : \mathbb{E}\left[\psi\left(\frac{|X|}{c}\right)\right] \leq 1 \right\}.
\end{equation}
A random variable $X$ is subgaussian if $\|X\|_{\psi_{e,2}} < \infty$ with $\psi_{e,2}(x) := e^{x^2} - 1.$
In particular, this implies that for some constants $C_1, C_2 > 0$,
\begin{equation}
{\P}\left(|X| \geq t\right) \leq 2 \exp\left( - \frac{C_1 t^2}{\|X\|_{\psi_{e,2}}^2} \right)
\quad \text{and} \quad
\mathbb{E}\left[|X|^p\right] ^{1/p} \le C_2 \sqrt{p}\, \|X\|_{\psi_{e,2}}, \quad p \geq 1.
\end{equation}
Likewise, $X$ is called subexponential if $\|X\|_{\psi_{e,1}} < \infty$ with $\psi_{e,1}(x) := e^{x} - 1.$
It implies that for some constants $C_1, C_2 > 0$,
\begin{equation}
{\P}\left(|X| \geq t\right) \leq 2 \exp\left( - \frac{C_1 t}{\|X\|_{\psi_{e,1}}} \right)
\quad \text{and} \quad
\mathbb{E}\left[|X|^p\right] ^{1/p} \le C_2 p\, \|X\|_{\psi_{e,1}}, \quad p \geq 1.
\end{equation}

Let  $(\mathcal{T}, d)$ denote a semimetric space and let $(X_t)_{t \in \mathcal{T}}$ be a stochastic process. The process has sub-$\psi$ increments if there exists $K > 0$ such that
\begin{equation}
\|X_t - X_s\|_{\psi} \le K d(t,s), \quad \forall\, t,s \in \mathcal{T}.
\end{equation}
Consider $\mathcal{T}_1 \subseteq \mathcal{T}$. For $\epsilon>0$, we denote by $N(\epsilon,\mathcal{T}_1,d)$ the covering number of the set $\mathcal{T}_1$ with respect to the semimetric $d$, that is, the minimal cardinality of a set of $\epsilon$-balls in the semimetric $d$ that covers $\mathcal{T}_1$. The logarithm of  $N(\epsilon,\mathcal{T}_1,d)$ is referred to as the metric entropy of $\mathcal{T}_1$. We define the Dudley integral 
\begin{equation}\label{equationDefinitionDI}
\text{DI}(\mathcal{T}_1,d,\psi) := \int_0^{\operatorname{diam}(\mathcal{T}_1)} \psi^{-1}\left(N(\epsilon,\mathcal{T}_1,d)\right)d\epsilon,
\end{equation}
where $\text{diam}(\mathcal{T}_1):= \sup_{x,x' \in \mathcal{T}_1}d(x,x')$  and $\psi^{-1}$ is the inverse function of $\psi$.
A key result for stochastic processes with controlled sub-$\psi$ increments is the following  maximal inequality  which provides uniform control over the increments of the process.

\begin{theorem}\label{TheoremMaximalInequality}(Adapted version of Theorem 8.4 in \cite{Kor08}: General maximal inequality for Orlicz norm)
Let $(X_t)_{t\in \mathcal{T}}$ be a separable stochastic process such that, for some
semimetric $d$ on $\mathcal{T}$ and some constant $r < \infty$,
\begin{equation}
\|X_s - X_t\|_{\psi} \leq r d(s,t),
\quad \forall\, s,t \in \mathcal{T}.
\end{equation}Then,
%for any $\eta, \delta > 0$,
%\[
%\|\sup_{\substack{s,t \in T \\ d(s,t) \le \delta}} |X(s) - X(t)|\; \|_{\psi_{e,2}}\le\;K\left[
%\int_{0}^{\eta}
%\sqrt{\log\bigl(D(\varepsilon, d)\bigr)}\, d\varepsilon
%\;+\;
%\delta\,\sqrt{\log\bigl(D_{2}(\eta, d)\bigr)}
%\right],
%\]
there exists a  constant $C < \infty$ depending only on  $\psi$ such that
\begin{equation}
\left\| \sup_{s,t \in \mathcal{T}} |X_s - X_t|
\;  \right\|_{\psi}  \le\;
2r C \int_{0}^{\operatorname{diam}(\mathcal{T})}
\psi^{-1}\left(N(\epsilon,\mathcal{T},d)\right)\, d\epsilon \lesssim r\operatorname{DI}\left(\mathcal{T}, d, \psi\right). 
\end{equation}
\end{theorem}

\subsection{Main assumptions}
We first introduce the main assumptions on our model.

\begin{assumption}\label{AssumptionsProcess} We assume that the particle system \eqref{eqIPS} satisfies the following assumptions:
\begin{itemize}
    \item [$\mathcal{A}_1)$] \label{AssumptionDrift} 
        The drift function  $b:{\R}\times\mathcal{P}_2({\R})\longrightarrow{\R}$ is {globally Lipschitz}. In particular, there exists a constant $B > 0$ such that 
        \begin{equation}
        |b(x, \mu)| \leq B(1 + |x| + W_2(\mu, \delta_0)),
        \end{equation}
        for all $x \in \R$ and $\mu\in\mathcal{P}_2(\R)$.
    \item [$\mathcal{A}_2)$] \label{AssumptionInteractingFunction} 
        The interaction function $\varphi: \mathbb{R} \to \mathbb{R}$ is globally Lipschitz continuous and bounded, satisfying that
        \begin{equation}
        |\varphi(x) - \varphi(y)| \leq L_\varphi |x - y|, \quad  |\varphi(x)| \le K_\varphi, \quad \forall \, x,y \in \R,
        \end{equation}
        for some finite constants $L_\varphi,  K_\varphi > 0$. 
    \item [$\mathcal{A}_3)$]\label{AssumptionNonnegativeVolatility} 
        The volatility of the system, $(\varphi \star \mu_{t}^{N})(x)$, is non-negative for all $t \in[0, T]$, $x \in \R.$
    \item[$\mathcal{A}_4)$] \label{AssumptionMomentsFinite}
        The initial conditions $(X_0^i)_{i=1}^N$ are i.i.d. with common law $\mu_0$ and admit moments of every order, i.e., $\forall p\geq1, \; i=1,\ldots,N,$ 
        \begin{equation}
        \mathbb{E}\left[|X_0^i|^p\right] < \infty.
        \end{equation}
\end{itemize}  
\end{assumption}
\noindent
 Under \cref{AssumptionsProcess}, both the interacting particle system \eqref{eqIPS} and its associated mean-field limit \eqref{eqMcKeanVlasovSDE} admit unique strong solutions. Moreover, since the initial conditions have moments of every order, standard moment estimates imply that, for every $p\geq1$,
\begin{equation}\label{eq:MomentsBoundforallP}
\sup_{1\leq i\leq N}\sup_{t\in[0,T]}{\E}\left[|X_t^i|^p\right]<\infty, \quad \sup_{t\in[0,T]}{\E}\left[|\overline{X}_t|^p\right]<\infty.
\end{equation}
For a detailed proof of this type of argument in the context of particle systems, we refer the reader to \cite[Lemma 5.1]{amo23}.

\section{Construction of the estimator and a general error bound}\label{sec:MinimizationCompactFE}
In this section, we introduce the proposed estimator and provide theoretical guarantees of its  statistical properties.
\subsection{Construction of the estimator}
We start by describing the framework of our estimation methodology. We consider the statistical scenario in which the trajectories of the $N$ particles satisfying \eqref{eqIPS} are discretely observed at times $t_j$, $j=0,\dots, n$, for $t_j = j\Delta_n$ with $\dn = T/n$ for a fixed $T>0$.
Our estimation approach is based on empirical risk minimization over a sequence of  approximation spaces. Under this methodology, we introduce the following sieve spaces $(S_{m})_{m\geq 1}$. Each $S_m$ is a  compact space of functions with respect to the supremum norm. Moreover, the elements of $S_m$ are  uniformly bounded and uniformly Lipschitz functions. In particular,  there exist constants $K_\varphi,L_\varphi>0$ such that, for every $m\ge1$ and every $f\in S_m$,
\begin{equation}
    \|f\|_\infty\le K_\varphi,
\qquad
\|f\|_{\mathrm{Lip}}\le L_\varphi.
\end{equation}
 When  convenient, we denote these spaces by $S_m(K_\varphi, L_{\varphi})$ to make explicit their dependence 
on the parameters $K_{\varphi}, L_{\varphi}$.
\noindent
Under our observation scheme, the oracle empirical risk minimization problem would be formulated as 
\begin{equation}\label{expressionIdealEstimator}
    \argmin{f \in S_N} \left\{ \frac{1}{Nn}\sum_{i=1}^N \sum_{j=0}^{n-1} \left( (f \star \mu_{t_j}^N)(X_{t_j}^i) - (\varphi \star \mu_{t_j}^N)(X_{t_j}^i)\right)^2\right\}.
\end{equation}
Directly linked to this formulation, we consider the following bilinear forms 
\begin{equation}\label{eqnaturalScalarProduct}
    \pnatural{f,g} = \frac{1}{nN}\sum_{i=1}^N \sum_{j=0}^{n-1} (f \star \mu_{t_j}^{N})(X_{t_j}^i)  (g \star \mu_{t_j}^{N})(X_{t_j}^i),
\end{equation}
\begin{equation}
    \pstarn{f,g} =  \frac{1}{n}\sum_{j=0}^{n-1} \int_{\R}(f \star \mu_{t_j})(x)(g \star \mu_{t_j})(x)d\mu_{t_j}(x),
\end{equation}
and define the associated norms $\nnatural{f}^2 := \pnatural{f,f}$, $\nstarn{f}^2 := \pstarn{f,f}$.
We emphasize that $\pnatural{\cdot,\cdot}$ is a random scalar product depending on the observations, while $\pstarn{\cdot,\cdot}$ is its deterministic counterpart, where $\mu_{t}$ refers to the law at time $t$ of $\overline{X}_t$, the solution of  \eqref{eqMcKeanVlasovSDE}.
We observe that, with respect to an element $f \in S_N$, minimizing \eqref{expressionIdealEstimator} is equivalent to minimizing 
\begin{equation}
    \nnatural{f - \varphi}^2 = \nnatural{\varphi}^2 + \nnatural{f}^2 -\frac{2}{nN}\sumsum (f \star \mu_{t_j}^{N})(X_{t_j}^i)(\varphi \star \mu_{t_j}^{N})(X_{t_j}^i).
\end{equation}
In this line, we define the oracle estimator and its deterministic counterpart as 
\begin{equation}\label{idealEstimatorChaos}
    \varphi_N^{\star} := \argmin{f \in S_N} \nnatural{f - \varphi}^2, 
\end{equation}\label{idealEstimatorMKV}
\begin{equation}
    \varphi^{\star}:= \argmin{f\in S_N} \nstarn{f-\varphi}^2.
\end{equation}
However, none of these estimators are directly computable just from the observations, since both involve the interaction function $\varphi$. In contrast, since we assume that our system satisfies \hyperref[AssumptionNonnegativeVolatility]{\(\mathcal{A}_3)\)}, we propose the following estimator 
\begin{equation}\label{expression3}
  \widehat{\varphi}_N:=   \argmin{f \in S_N} \left\{ \frac{1}{Nn}\sum_{i=1}^N \sum_{j=0}^{n-1} \left( (f \star \mu_{t_j}^N)(X_{t_j}^i)- \sqrt{\frac{\pi}{2\dn}}\left|X^i_{t_j+1} - X^i_{t_j}\right|\right)^2\right\}.
\end{equation}
To provide some intuition for the proposed estimator, we note that $\widehat{\varphi}_N$ can be interpreted as a computable counterpart of the estimator one would have obtained via the minimization in \eqref{expressionIdealEstimator}. Indeed, at high-frequency the drift component of the diffusion model is asymptotically negligible and we deduce, conditionally on $\mathcal{F}_{t_j}$,
\[
\Delta_n^{-1/2} (X^i_{t_j+1} - X^i_{t_j}) \approx 
\Delta_n^{-1/2} (\varphi \star \mu_{t}^{N})( X^i_{t_j}) 
(W^i_{t_j+1} - W^i_{t_j}) \sim \mathcal{N}\left(0,
(\varphi \star \mu_{t}^{N})^2( X^i_{t_j})\right).
\]
Consequently, we obtain the approximation 
\[
\E\left[\sqrt{\frac{\pi}{2\dn}}\left|X^i_{t_j+1} - X^i_{t_j}\right| | \mathcal{F}_{t_j}\right]\approx 
(\varphi \star \mu_{t}^{N})( X^i_{t_j}), 
\]
which provides the heuristic justification for the construction in \eqref{expression3}. A rigorous study of the estimator is provided in the following sections.
Expanding \eqref{expression3}, we observe that, minimizing with respect to $f \in S_N$ is equivalent to minimizing with respect to the same function space the following random functional
\begin{align}\label{eqGammaN}
    \gamma_N(f) &= \frac{1}{Nn}\sum_{i=1}^N \sum_{j=0}^{n-1} \left( (f \star \mu_{t_j}^N)^2(X_{t_j}^i)  - \sqrt{\frac{2\pi}{\dn}}(f \star \mu_{t_j}^N)(X_{t_j}^i)\left|X^i_{t_j+1} - X^i_{t_j}\right|\right).
\end{align}\label{eqGamma}
Consequently, $\widehat{\varphi}_N$ admits the equivalent expression
\begin{equation}\label{eqEstimator}
    \widehat{\varphi}_N = \argmin{f \in S_N} \gamma_N(f).
\end{equation}
Throughout the article, we establish several theoretical results concerning the statistical performance of the estimator $\widehat{\varphi}_N$. A first key observation is that, introducing the quantity 
\begin{equation}\label{eqGammaprime}
    \gamma'_N(f) := \frac{2}{nN}\sumsum (f \star \mu_{t_j}^{N})(X_{t_j}^i)\left(\sqrt{\frac{\pi}{2\dn}}\left|X_{t_j+1}^i - X_{t_j}^i\right| -(\varphi \star \mu_{t_j}^{N})(X_{t_j}^i) \right),
\end{equation}
we have that 
\begin{equation}\label{eq1}
    \nnatural{f-\varphi}^2 = \gamma_N(f) +  \gamma'_N(f)  + \nnatural{\varphi}^2.
\end{equation}
This identity provides an intuitive justification for the proposed estimator. Minimizing the ideal empirical risk $\nnatural{f-\varphi}^2$ corresponds to solving the oracle problem \eqref{idealEstimatorChaos}, while minimizing the functional contrast $\gamma_N(f)$ differs from this ideal criterion only up to the remainder term $\gamma'_N(f)$, which in some sense can be interpreted as  the empirical discrepancy between the computable contrast and the oracle criterion.

\subsection{Main results}
Our first result  provides a non-asymptotic $L_p$-error bound  of the proposed estimator under the intrinsic seminorm associated with the geometry of the statistical model.

\begin{theorem}\label{OracleCompactSpace}
      Under \cref{AssumptionsProcess}, for all $p\geq 2,$ it holds that 
 \begin{align}
          \E \left[\nstarn{{{\widehat{\varphi}}}_N - \varphi}^p \right]^{1/p} \lesssim  \inf_{f\in S_N}\nstarn{f - \varphi} 
          + C_{p,K_{\varphi},L_{\varphi}}(1 + \mathfrak{D})^{1/2}(N^{-1/(4p)} + \dn^{1/4}),
    \end{align} 
    where $C_{p,K_{\varphi},L_{\varphi}}^2 = p(1 + K_{\varphi}L_{\varphi})(1 + K_{\varphi})$ and $\mathfrak{D} = \text{DI}\left( {S}_N(2K_{\varphi},2L_{\varphi}), \|\cdot\|_{\infty}, \psi_{e,1} \right).$ 
\end{theorem}
\noindent

\begin{proof}By definition of ${\widehat{\varphi}}_N$ and $\varphi^{\star}$, it yields that  ${\gamma}_N({\widehat{\varphi}}_N) \leq {\gamma}_N(\varphi^{\star})$. From \eqref{eq1}, we have 
    \begin{equation}
        \nnatural{{\widehat{\varphi}}_N - \varphi}^2 - {\gamma}'_N({\widehat{\varphi}_N})  \leq \nnatural{{\varphi}^{\star} - \varphi}^2 - {\gamma}'_N(\varphi^{\star}). 
    \end{equation}
Adding and subtracting $\nstarn{{\widehat{\varphi}}_{N} - \varphi}^2$ and $\nstarn{\varphi^{\star} - \varphi}^2$, we obtain 
\begin{align}
    \nstarn{{\widehat{\varphi}}_{N} - \varphi}^2 &\leq \nstarn{\varphi^{\star} - \varphi}^2 + \left(\nnatural{\varphi^{\star} - \varphi}^2 - \nstarn{\varphi_{\star} - \varphi}^2 + \nstarn{{\widehat{\varphi}}_{N} - \varphi}^2 - \nnatural{{\widehat{\varphi}}_N - \varphi}^2 \right)\\
    &+ \left({\gamma}'_N({\widehat{\varphi}_N})- {\gamma}'_N(\varphi^{\star})\right) \\
    &\leq \nstarn{\varphi^{\star} - \varphi}^2 + 2 \left( \sup_{f \in S_N}\left|\nnatural{f-\varphi}^2 - \nstarn{f - \varphi}^2\right| + \sup_{f \in S_N}\left|{\gamma}_{N}'(f)\right|\right).
\end{align}
It follows that
\begin{align}
          \E \left[\nstarn{{\widehat{\varphi}}_N - \varphi}^p \right]^{1/p} &\leq \inf_{f\in S_N}\nstarn{f - \varphi} + \sqrt{2}\mathbb{E} \left[ \sup_{f \in S_N}\left|\nnatural{f-\varphi}^2 - \nstarn{f - \varphi}^2\right|^{p/2} \right]^{1/p}   \\ \label{eq:MainTheoremTwoErrors}
        &+\sqrt{2}\mathbb{E} \left[ \sup_{f \in S_N}\left|{\gamma}'_N(f)\right|^{p/2} \right]^{1/p}.
    \end{align} 
From Proposition \ref{PropositionLpControlDifferenceNorms}, for all $p\geq 2,$ we have 
\begin{equation}
    \mathbb{E} \left[ \sup_{f \in S_N}\left|\nnatural{f-\varphi}^2 - \nstarn{f - \varphi}^2\right|^{p/2} \right]^{1/p} \lesssim \frac{\sqrt{K_{\varphi}L_{\varphi}}}{N^{1/(4p)}}.
\end{equation}
On the other hand, from Proposition  \ref{PropositionLpControlGamma'}, for $p\geq 2$ we obtain 
\begin{equation}
    \mathbb{E} \left[ \sup_{f \in S_N}\left|{\gamma}'_N(f)\right|^{p/2} \right]^{1/p} \lesssim {p^{1/2} \dn^{1/4}(1+L_{\varphi}K_{\varphi})^{1/2} \left( \text{DI}\left( {S}_N(2K_{\varphi}, 2L_{\varphi}), \|\cdot\|_{\infty}, \psi_{e,1} \right) + K_{\varphi}\right)^{1/2}}.
\end{equation}
Combining the two estimates 
together with \eqref{eq:MainTheoremTwoErrors} we obtain
the claimed bound.
\end{proof}

\noindent
The bound in \eqref{eq:MainTheoremTwoErrors} admits a natural interpretation in terms of the three successive approximations underlying the estimation procedure. The first contribution is the approximation error
\begin{equation}
    \inf_{f\in S_N}\nstarn{f - \varphi},
\end{equation}
which is induced by restricting the estimator to the finite-dimensional sieve space $S_N$. This term depends on the approximation properties of the chosen sieve and a detailed analysis of this term for multiresolution wavelet spaces is provided in Section \ref{sec:ApproximationError}. 

The second contribution is what we define as the mean-field empirical error
\begin{equation}\label{eq:EmpiricalMFerror}
     \sup_{f \in S_N}\left|\nnatural{f-\varphi}^2 - \nstarn{f - \varphi}^2\right|.
\end{equation}
This term measures the discrepancy between the deterministic norm $\nstarn{\cdot}$, associated with the McKean--Vlasov limit and the empirical  norm $\nnatural{\cdot}$, which is based on the empirical measure $\mu_t^N$ of the observed interacting particle system. It thus quantifies the error in approximating the mean-field geometry by its finite-particle counterpart. Proposition \ref{PropositionLpControlDifferenceNorms} provides a detailed control of this term by exploiting quantitative propagation of chaos estimates, yielding the contribution $N^{-1/(4p)}$ in Theorem \ref{OracleCompactSpace}.

The third contribution is what we will refer to as the empirical contrast error
\begin{equation}\label{eq:EmpiricalContrastError}
 \sup_{f \in S_N}\left|{\gamma}'_N(f)\right|. 
\end{equation}
This term quantifies the discrepancy between the ideal empirical risk and the computable contrast as  made explicit in \eqref{eq1}. It combines the error resulting from approximating the local volatility by the observed increments over intervals of length $\dn$ and the martingale fluctuations generated by the Brownian noise. Proposition \ref{PropositionLpControlGamma'} controls both components, with the discretization component dominating the martingale one and therefore determining the $\dn^{1/4}$ contribution in Theorem \ref{OracleCompactSpace}.

\begin{remark}
    The control of the mean-field empirical error \eqref{eq:EmpiricalMFerror} and the empirical contrast error \eqref{eq:EmpiricalContrastError} differs fundamentally from previous approaches that are based on a change-of-measure argument, introduced in \cite{della22} and used, for instance, in \cite{bel24}. There, estimates are first established for $N$ i.i.d. particles solving the McKean--Vlasov equation and subsequently transferred to the interacting system. This strategy is not available in our setting, since the diffusion coefficients in \eqref{eqIPS} and \eqref{eqMcKeanVlasovSDE} do not coincide due to the convolution of the interaction function with the respective measures. We instead control these terms directly using quantitative propagation-of-chaos estimates.
\end{remark}
\noindent
The following corollary illustrates the two possible asymptotic regimes according to the relative growth of $N$ and $n$.

\begin{corollary}\label{CorollaryRateLpMCV}Let us suppose that for all $\epsilon >0$, $N\in \N$ and some $C_1>0$,
\begin{equation}
N(\epsilon, S_N(2K_{\varphi},2L_{\varphi}),\|\cdot\|_{\infty}) \leq C_1\epsilon^{-D_{N}},
\end{equation}  where $(D_N)_{N\geq1}$ is an increasing sequence
satisfying $D_N^2(N^{-1/p} \vee \dn ) \rightarrow 0$. Then, for every $p\geq2$, the following two
regimes hold:
\begin{itemize}
    \item[\textup{(i)}] If $\dn N^{1/p} \rightarrow 0$,  the discretization error is asymptotically negligible relative
    to the mean-field empirical error and
    \begin{equation}
        \E \left[\nstarn{{\widehat{\varphi}}_N - \varphi}^p \right]^{1/p} \leq \inf_{f\in S_N}\nstarn{f - \varphi} + C_{p,K_{\varphi,L_{\varphi},C_1}}\frac{D_N^{1/2}}{N^{1/(4p)}}.
    \end{equation} 
    \item[\textup{(ii)}] If $\dn N^{1/p} \rightarrow \infty$,  the discretization error asymptotically dominates
    the mean-field empirical error, and
    \begin{equation}
        \E \left[\nstarn{{\widehat{\varphi}}_N - \varphi}^p \right]^{1/p} \leq \inf_{f\in S_N}\nstarn{f - \varphi} + C_{p,K_{\varphi,L_{\varphi},C_1}}\dn^{1/4}D_N^{1/2}.
    \end{equation} 
\end{itemize}
\end{corollary}
\begin{proof}From the definition of the Dudley integral in \eqref{equationDefinitionDI}, we observe that 
\begin{equation}
   \mathfrak{D}  = \int_0^{\text{diam}(S_N(2K_{\varphi},2L_{\varphi}))}{\log\left(1 + N(\epsilon,{S}_N(2K_{\varphi},2L_{\varphi}), \| \cdot\|_{\infty})\right)}d\epsilon \lesssim {D_N}.
\end{equation}
Hence, Theorem \ref{OracleCompactSpace} gives
\[
\E\left[\nstarn{{\widehat{\varphi}}_N-\varphi}^p\right]^{1/p}
\lesssim
\inf_{f\in S_N}\nstarn{f-\varphi}
+D_N^{1/2}\left(N^{-1/(4p)}+\dn^{1/4}\right).
\]
If $\dn N^{1/p}\to0$, then $\dn^{1/4}=o(N^{-1/(4p)})$, yielding
\textup{(i)}. If $\dn N^{1/p}\to\infty$, then
$N^{-1/(4p)}=o(\dn^{1/4})$, yielding \textup{(ii)}. Finally, the condition $D_N^2(N^{-1/p} \vee \dn ) \rightarrow 0$ ensures that both stochastic contributions vanish asymptotically.
\end{proof}

\begin{remark}
The polynomial entropy assumption $N(\epsilon, S_N(2K_{\varphi},2L_{\varphi}),\|\cdot\|_{\infty}) \leq C_1\epsilon^{-D_{N}}$ is the classical Euclidean entropy condition. It is satisfied by a broad class of finite-dimensional approximation spaces, in which case the exponent $D_N$ typically coincides with the dimension of the approximating space $S_N$.
\end{remark}

\begin{remark}
We now comment on the condition regarding the discretization step, which simplifies to $\Delta_n N^{1/2} \to 0$ in our main case of interest, $p=2$. It is natural to compare this requirement with existing results in the literature. To our knowledge, estimation based on discrete observations of interacting particles has been studied primarily in the context of parameter estimation, such as in \cite{amo23, amo26}, where the stronger condition $N \Delta_n \to 0$ is required. 

Both works note that this condition may not be optimal and could potentially be relaxed through alternative approaches. This parallels the ergodic setting, where the original condition $T \Delta_n \to 0$ was successively relaxed to $T \Delta_n^2 \to 0$ \cite{Yos92}, and subsequently to general polynomial rates in $\Delta_n$ \cite{kessler} (see also \cite{amorino2020contrast, Shimizu} for SDEs with jumps). 

However, the condition arising in our context is not only weaker, but also fundamentally different in spirit. While previous conditions are indispensable for establishing the target asymptotic properties of the estimator, our setting delineates two distinct regimes: even when the discretization condition fails, our results remain valid, albeit at a slower rate of convergence.
\end{remark}

\noindent
    The  convergence rate related to the stochastic error term in Corollary \ref{CorollaryRateLpMCV} is not expected to be optimal. This suboptimality is mainly due to the general techniques used to control the empirical error term, which do not exploit additional structural properties of the interacting particle system and are designed to hold under rather mild assumptions. A more refined analysis, or stronger assumptions on the model, could potentially lead to sharper rates. Nevertheless, as will be seen in Section \ref{sec:L2ERROR}, this suboptimality does not impact the final convergence rate of the estimator with respect to the $L^2$-norm, which is the main objective of this work.

\section{Approximation error over wavelet spaces }\label{sec:ApproximationError}
The estimation error obtained in Theorem \ref{OracleCompactSpace} naturally decomposes into a stochastic term and an approximation term. The latter depends on the approximation properties of the sieve space $S_N(K_{\varphi}, L_{\varphi})$. In this section, we study this approximation error for a class of multiresolution wavelet spaces. While Corollary \ref{CorollaryRateLpMCV} holds for a general $(\mu_t)_{t\in [0,T]}$, to obtain meaningful quantitative bounds for the approximation term, we restrict our study to a class of models for which the density $\mu_t$ satisfies  Gaussian upper and lower bounds. To derive these estimates, we impose the following additional assumptions.

\begin{assumption}\label{AssumptionProcessII}
    Suppose that  \cref{AssumptionsProcess} holds. In addition, we assume the following:
    \begin{itemize}
        \item [$\mathcal{B}_1)$] The drift function $b:{\R}\times\mathcal{P}_2({\R})\longrightarrow{\R}$ is bounded.
        
        %\item [$\mathcal{B}_2)$]The volatility function of \eqref{eqMcKeanVlasovSDE}  and \eqref{eqIPS} is bounded away from 0. It means that  there exists a constant $\mathfrak{m} > 0$ such that $\forall x \in \R, t\in [0,T]$,
       % \begin{equation}
         %   (\varphi \star \mu_t)(x)  \geq \mathfrak{m} \quad \text{and}\quad
        % (\varphi \star \mu^N_t)(x) \geq \mathfrak{m}.
        %\end{equation}
        \item[$\mathcal{B}_2)$] The diffusion coefficient of the McKean--Vlasov equation \eqref{eqMcKeanVlasovSDE} is uniformly bounded away from zero. More precisely, there exists a constant $\mathfrak{m}>0$ such that $\forall x \in\mathbb{R}, t\in[0,T]$,
         \begin{equation}   
         (\varphi \star \mu_t)(x)\geq \mathfrak{m}. 
        \end{equation}
    \end{itemize}
\end{assumption}
\noindent
We remark that condition $\mathcal{B}_2$ is automatically satisfied if $\varphi \geq \mathfrak{m}$.

\begin{assumption}\label{AssumptionInitialDistribution}
    Assume that the initial distribution $\mu_0$ admits Gaussian upper and lower bounds. More precisely,  there exist positive constants $c_1,c_2,\mathcal{C}_1, \mathcal{C}_2$ such that for all $x \in \R$,
\begin{equation}
    c_1 \phi_{\mathcal{C}_1}(x) \leq \mu_0(x) \leq c_2 \phi_{\mathcal{C}_2}(x). 
\end{equation}
\end{assumption}
\noindent
Under these assumptions, the density $\mu_t$ admits Gaussian upper and lower bounds uniformly in $t \in [0,T]$, as stated in the following lemma.

\begin{lemma}\label{LemmaSubgaussianDistribution} Under \cref{AssumptionProcessII} and \cref{AssumptionInitialDistribution}, there exist positive constants $g_1,g_2, \sigma_1, \sigma_2 $ such that $\forall x \in \R$ and  $t\in [0,T];$
    \begin{equation}
        g_1\phi_{\sigma_1}(x) \leq \mu_t(x) \leq g_2\phi_{\sigma_2}(x),
    \end{equation}
    where $\phi_{\sigma_i}:\R \rightarrow \R$ denotes the density  of the Gaussian distribution $\mathcal{N}(0,\sigma_i^2)$.
\end{lemma}

\begin{proof} 
    Let us consider the  SDE
    \begin{equation}\label{eqdefinedSDE}
        dY_t = b(Y
        _t, \mu_t)dt + (\varphi \star \mu_t)(Y_t)dB_t ,\quad Y_0 \sim \mu_0,
    \end{equation}
    which is obtained by freezing the densities $\mu_t$ in the McKean-Vlasov SDE defined in \eqref{eqMcKeanVlasovSDE}. Note that $Y_t$ has the same law as  $\overline{X}_t$, the solution of \eqref{eqMcKeanVlasovSDE}.  Let us consider the associated infinitesimal generator of the diffusion process
    \begin{equation}
        \mathcal{L} = b(x,\mu_t)\frac{\partial}{\partial x} + \frac{1}{2}a(t,x)\frac{\partial^2}{\partial x^2},
    \end{equation}
    where $a(t,x) = (\varphi \star \mu_{t})^2(x)$. Under \cref{AssumptionProcessII},  the SDE \eqref{eqdefinedSDE} has a bounded and uniformly Lipschitz drift and the function $a(t,x)$ is uniformly elliptic. Therefore, if we denote by $\rho(t,y;0,x)$ the transition density of the process, by Aronson's estimates (Theorem 1 in \cite{Aro67}),   for all $x,y\in \R$ and $t\in(0,T]$, there exist positive constants $C_1,C_2,C_3$ and $C_4$ such that 
\begin{equation}\label{eqAroson}
\frac{C_1}{t^{1/2}}\exp\left(-\frac{|x-y|^2}{C_2t}\right)
\leq
\rho(t,y; 0,x)
\leq
\frac{C_3}{t^{1/2}}\exp\left(-\frac{|x-y|^2}{C_4t}\right).
\end{equation}
Integrating these bounds in $x$ with the initial distribution $\mu_0$, we obtain that 
\begin{equation}
\int_{\R}\frac{C_1}{t^{1/2}}\exp\left(-\frac{|x-y|^2}{C_2t}\right)\mu_0(x)dx
\leq
\rho_t(y)
\leq \int_{\R}
\frac{C_3}{t^{1/2}}\exp\left(-\frac{|x-y|^2}{C_4t}\right)\mu_0(x)dx.
\end{equation}
Finally, using the bounds for $\mu_0(x)$ in \cref{AssumptionInitialDistribution}, we conclude our lemma.
\end{proof}

\noindent
Following the same approach as in \cite{bel24}, in order to obtain explicit approximation rates, we consider sieve spaces generated by multiresolution wavelet expansions. These spaces provide a classical and widely used framework for nonparametric approximation, combining spatial localization with multiscale representations of the target function. Moreover, the decay of the associated wavelet coefficients naturally characterizes the regularity of the function. For further background on multiresolution analysis, we refer to \cite[Chapter 4]{Gin15}.

Let  $\xi \in L^2(\R)$ be the scaling function and let $\psi \in L^2(\R)$ be the associated mother wavelet, and  assume that $\xi $ and $\psi$ are compactly supported and sufficiently smooth. In particular, their Fourier transform satisfies   
\begin{equation}\label{eqw}
   |\widehat{\xi}(x)| + |\widehat{\psi}(x)| \leq C\left(1 + |x|\right)^{-w},  \hspace{0.5cm} x\in\R,
\end{equation}
 for some constants $C,w > 0.$ We note that these assumptions are satisfied, for instance, by sufficiently
regular Daubechies wavelets. For $j \geq 0$ and $k \in \Z,$ we define the dilated and translated functions
\begin{equation}
    \xi_{jk}(x) := 2^{j/2}\xi(2^jx - k ), \hspace{1cm} \psi_{jk}(x) := 2^{j/2}\psi(2^jx - k ).
\end{equation}
 The  family  $\{\xi_{0k}, \psi_{jk}\}_{j\geq0}^{k\in Z}$ forms an orthonormal basis in $L^2(\R)$ (see \cite[Eq. (4.160)]{Gin15}). Hence, every $f\in L^2(\R)$ admits the expansion
\begin{equation}
    f(x) = \sum_{k \in \Z}\tilde{c}_{0k}\hspace{0.05cm}\xi_{0k}(x) + \sum_{j \geq 0, k\in \Z}c_{jk}\hspace{0.05cm}\psi_{jk}(x) = \sum_{j\geq-1, k\in \Z} c_{jk}\hspace{0.05cm}\psi_{jk}(x),
\end{equation}
with the convention that $c_{-1k} = \tilde{c}_{0k}$ and $\psi_{-1k} = \xi_{0k}.$
Motivated by this representation, for \(C,r_1,r_2>0\), we introduce the class
\begin{equation}
    S(C,r_1,r_2) := \left\{f:\R \rightarrow\R \big| f(x) = \sum_{j\geq -1}\sum_{k\in \Z}c_{jk}\hspace{0.05cm}\psi_{jk}(x),\hspace{0.2cm}
    |c_{jk}| \leq C2^{-j(r_2 +1/2)}(1 + |k|)^{-r_1}\right\}.
\end{equation}
The parameters $r_1$ and $r_2$ respectively control the spatial localization and the smoothness of the functions through the decay of the coefficients in the spatial and resolution indices. The constant $C>0$ provides a uniform bound on the size of the class. In particular, $S(C,r_1,r_2)$ can be viewed as a localized counterpart of the classical Besov space $B_{\infty,\infty}^{r_2}$, where the additional decay in the spatial index enforces localization at infinity. To construct the sieve spaces used in the estimation procedure, we consider finite-dimensional approximations of $S(C,r_1,r_2)$  obtained by truncating both the resolution level and the spatial support of the wavelet expansion. More precisely, for $A>0$ and $J \in \N$, we define

\begin{gather}
    S(A,J,C,r_1,r_2) := \Big\{f:\R \rightarrow\R \big| f(x) = \sum_{j\geq -1}^{J-1} \sum_{\substack{k\in\mathbb Z\\ \operatorname{supp}(\psi_{jk})\subseteq [-A,A]}}c_{jk}\hspace{0.05cm}\psi_{jk}(x), \\
   \hspace{1cm} |c_{jk}| \leq C2^{-j(r_2 +1/2)}(1 + |k|)^{-r_1}\Big\}.
\end{gather}
The parameter $J$ determines the finest resolution level included in the approximation, whereas $A$ controls its spatial localization by restricting the wavelet basis to functions supported in $[-A,A]$.
For the approximation analysis, it will be convenient to work with a class characterized directly by analytic properties of the functions rather than by their wavelet coefficients. To this end, we introduce the following regularity class.
\begin{gather}
    \mathcal{A}(K,r_1,r_2) := \Big\{f:\R \rightarrow \R \hspace{0.1cm }\big| \hspace{0.1cm }\|f\|_{L^1(\R)} + \|f\|_{L^2(\R, |x|^2dx)} + \|f\|_{\infty} \leq K  \\
    \hspace{2.cm}|f(x)| \leq K(1 \wedge |x|^{-r_1}), \hspace{0.2cm} |\widehat{f}(x)| \leq K(1 \wedge |x|^{-r_2})\Big\}.
\end{gather}
The following lemma establishes a direct link between the wavelet classes introduced above and the class $\mathcal{A}(K,r_1,r_2)$.
\begin{lemma}\label{lemmaInclusionApproximationSpaces}(Lemma 3.2 of \cite{bel24})
    Assume that \begin{equation}\label{conditionr1r2}
     r_1 > \frac{3}{2}, \hspace{1cm }1< r_2 < w
\end{equation}
for $w > 0$ defined in \eqref{eqw}. Then, for every $C > 0,$ there exists a constant $K>0$ such that 
    \begin{equation}
         S(C,r_1,r_2) \subset\mathcal{A}(K,r_1,r_2),  \hspace{0.5cm}  S(A,J,C,r_1,r_2)\subset\mathcal{A}(K,r_1,r_2),
    \end{equation}
    for all $A>0$ and $J \in \N$.
\end{lemma}
\noindent
The next proposition quantifies the approximation error over the sieve spaces $S(A,J,C,r_1,r_2)$.

\begin{proposition}\label{PropositionInfimumApproximationError} 
Assume that $\varphi \in S(C,r_1,r_2)$ and condition \eqref{conditionr1r2} is satisfied. Then,  
under \cref{AssumptionProcessII} and \cref{AssumptionInitialDistribution}, it holds that
\begin{equation}
    \inf_{f \in  S(A,J,C,r_1,r_2)} \nstarn{f - \varphi} \lesssim \frac{K}{\sqrt{A}}\exp\left(-\frac{A^2} {16\sigma_2^2}\right) + 2^{-J(r_2 +1)},
\end{equation}
    where $\sigma_2^2$ is introduced in Lemma \ref{LemmaSubgaussianDistribution}.
\end{proposition}
\noindent
The detailed proof of Proposition \ref{PropositionInfimumApproximationError} follows closely the proof of Proposition 3.3 in \cite{bel24} and can be consulted in Section \ref{proof:PropApproximationError}.

 The following corollary characterizes the performance of the proposed estimator under the  $\|\cdot\|_{\star,n}$-norm.

\begin{corollary}\label{CorollaryControlApproximationError}
Under \cref{AssumptionProcessII}, \cref{AssumptionInitialDistribution}, and the conditions in Corollary \ref{CorollaryRateLpMCV}, assume that
$\varphi\in S(C,r_1,r_2)$ for $r_1,r_2>0$ satisfying \eqref{conditionr1r2}.
Define
\begin{equation}
    \eta:=N^{-1/p}\vee\dn.
\end{equation}
For $\eta>0$, set
\begin{equation}\label{setAeta2Jeta}
    A_{\eta}=2\sigma_2\sqrt{\log(\eta^{-1})},
    \qquad
    2^{J_{\eta}}
    =O\left(
    \frac{\eta^{-1}}{\log(\eta^{-1})}
    \right)^{\frac{1}{4(r_2+3/2)}},
\end{equation}
and let $S_N=S(A_{\eta},J_{\eta},C,r_1,r_2).$
Then, the estimator \eqref{eqEstimator} satisfies the following rates:
\begin{itemize}
    \item[\textup{(i)}] If $\dn N^{1/p}\to0$, 
\begin{equation}\label{eq:rate1Section3}
        \E\left[
        \|\widehat{\varphi}_N-\varphi\|_{\star,n}^p
        \right]^{1/p}
        \lesssim
        \left(
        \frac{\log N}{N^{1/p}}
        \right)^{\alpha(r_2)}.
    \end{equation}

    \item[\textup{(ii)}] If $\dn N^{1/p}\to\infty$, \begin{equation}\label{eq:rate2Section3}
        \E\left[
        \|\widehat{\varphi}_N-\varphi\|_{\star,n}^p
        \right]^{1/p}
        \lesssim
        \left(
        \dn\log(\dn^{-1})
        \right)^{\alpha(r_2)},
    \end{equation}
\end{itemize}
where $\alpha(r_2)=({r_2+1})/({4r_2+6}).$
\end{corollary}

\begin{proof}
Since $S_N=S(A_{\eta},J_{\eta},C,r_1,r_2)$, the effective dimension
of $S_N$ has order $A_{\eta}2^{J_{\eta}}$. Combining Corollary
\ref{CorollaryRateLpMCV} and Proposition
\ref{PropositionInfimumApproximationError}, and choosing 
$A_{\eta}$ as in \eqref{setAeta2Jeta}, we obtain
\begin{equation}
\E\left[
\|\widehat{\varphi}_N-\varphi\|_{\star,n}^p
\right]^{1/p}
\lesssim
\frac{\eta^{1/4}}
{\log(\eta^{-1})^{1/4}}
+
2^{J_{\eta}/2}
\log(\eta^{-1})^{1/4}\eta^{1/4} +2^{-J_{\eta}(r_2+1)}.
\end{equation}
Choosing $J_{\eta}$ as in \eqref{setAeta2Jeta} and balancing the last
two terms yields
\begin{equation}
\E\left[
\|\widehat{\varphi}_N-\varphi\|_{\star,n}^p
\right]^{1/p}
\lesssim
\left(
\eta\log(\eta^{-1})
\right)^{(r_2 + 1)/(4r_2 + 6)}.
\end{equation}
Finally, the two regimes and their respective bounds follow directly from
Corollary \ref{CorollaryRateLpMCV}.
\end{proof}

\noindent
The resulting rates in Corollary \ref{CorollaryControlApproximationError}
are polynomial in the effective scale $N^{-1/p}\vee\dn$, up to a logarithmic
factor. The exponent $\alpha(r_2)$ is determined by the resolution
regularity parameter $r_2$, which governs the decay of the wavelet
approximation error across resolution levels. Larger values of $r_2$
correspond to smoother interaction functions and therefore lead to faster
convergence rates. In particular, the exponent converges to $1/4$ as
$r_2\to\infty$. Thus, for very smooth interaction functions, the
approximation error becomes negligible and the rate approaches the
benchmark
\begin{equation}
   \E\left[
   \|\widehat{\varphi}_N-\varphi\|_{\star,n}^p
   \right]^{1/p}
   \lesssim
   \left(
   (N^{-1/p}\vee\dn)
   \log\big((N^{-1/p}\vee\dn)^{-1}\big)
   \right)^{1/4},
\end{equation}
whose exponent coincides with the stochastic rate in Theorem
\ref{OracleCompactSpace}.

\section{Analysis of the $L^2$-error}\label{sec:L2ERROR}
In this section, we study the error of the estimator defined in \eqref{eqEstimator} with respect to the classical $L^2(\R)$-norm, in contrast to the result provided in Corollary \ref{CorollaryRateLpMCV}, where the error is evaluated under the natural $\nstarn{\cdot}$-norm associated with the problem. We first note that the $\nstarn{\cdot}$-norm can be bounded from above by the $L^2(\R)$-norm. If the density $\mu_t$ satisfies $\sup_{0\leq t\leq T}\|\mu_t\|_{\infty}<\infty,$ then
\begin{equation}
    \nstarn{f}^2 \leq \frac{1}{n}\sum_{j=0}^{n-1}\|f \star \mu_{t_j}\|_{L^2}^2\|\mu_{t_j}\|_{\infty}
    \leq \|f\|_{L^2}^2 \sup_{0\leq t\leq T}\|\mu_t\|_{\infty}.
\end{equation}
However, the two norms are not equivalent. The following elementary example, originally presented in \cite{bel24}, illustrates this fact. Consider $\mu_t=\mu$, where $\mu$ is a standard Gaussian density for all $t\in[0,T]$. Let
$f(x)=\mathds{1}_{[2M,3M]}(x)$ for $ M>0.$ It is easy to show that $\|f\|_{L^2}^2=M$, while
\begin{equation}
    \nstarn{f}^2
    \leq
    \int_{-M}^{M}(f\star\mu)^2(x)\mu(x)\,dx
    +
    \int_{[-M,M]^C}(f\star\mu)^2(x)\mu(x)\,dx
    \leq
    c_1\exp(-c_2M),
\end{equation}
for some constants $c_1,c_2>0$. Therefore, as $M\to\infty$, we have $\|f\|_{L^2}^2\to\infty$, while $\nstarn{f}^2\to0$ exponentially fast.
Although the two norms are not equivalent, the following proposition shows that the $L^2(\R)-$norm  can still be controlled by the intrinsic $\| \cdot \|_{\star,n}-$norm, up to explicit tail terms depending on the function  and $\mu_t$.
\noindent

\begin{proposition}\label{PropositionRelationL2NStarn}
Let $\mu_t:\R \rightarrow \R$ be a density function, $f : \mathbb{R} \to \mathbb{R}$ be a bounded function in $L^1(\mathbb{R})$ and $\widehat{f}$ its  Fourier transform. Define
\begin{equation}
m_{R_1}^{t} := \inf_{x \in B_{R_1}} |\mu_{t}(x)|,
\hspace{0.5cm} \widehat{m}_{R_2}^{t}:= \inf_{x \in B_{R_2}} |\widehat{\mu}_{t}(x)|^2, \hspace{0.5cm}   \hspace{0.5cm} M_{R_1R_2}^n:=\frac{1}{n}\sum_{j=0}^{n-1}m_{R_1}^{t_j}\widehat{m}^{t_j}_{R_2},
\end{equation}
for any  $R_1,R_2>0$ and any collection of points $\{t_j\}_{j=0}^{n-1}$.
Then, 
\begin{equation}
\|f\|_{L^2}^2 \leq \left(M_{R_1R_2}^n\right)^{-1}\nstarn{f}^2  + \left(M_{R_1R_2}^n\right)^{-1}\frac{1}{n}\sum_{j=0}^{n-1}m_{R_1}^{t_j}\| f \star \mu_{t_j}\|^2_{L^2(B^{\mathcal{C}}_{R_1})}+ \|\widehat{f}\|^2_{L^2(B^{\mathcal{C}}_{R_2})}.
\end{equation}
\end{proposition}
\noindent
The detailed proof of Proposition \ref{PropositionRelationL2NStarn} can be found in  Section \ref{proof:PropositionRelationL2NStarn}.

Proposition \ref{PropositionRelationL2NStarn} provides a general comparison between the  $\|\cdot\|_{\star,n}$ and the classical $L^2(\mathbb{R})$-norm, valid for arbitrary density functions. The price for this level of generality is that the quantities $m_{R_1}^{t_j}$ and $\widehat{m}_{R_2}^{t_j}$ cannot, in general, be computed explicitly. However, inspired by the proof of Proposition 4.1 in \cite{bel24} and following the approach of Section \ref{sec:ApproximationError}, we use the bounds established in Lemma \ref{LemmaSubgaussianDistribution}  to derive explicit lower bounds for these quantities, leading to the following corollary.

\begin{corollary}\label{CorollaryL2Bound}
Let $f$ and $\mu_t$ as in Proposition \ref{PropositionRelationL2NStarn}. Consider \cref{AssumptionProcessII} and \cref{AssumptionInitialDistribution}. Assume  that there exists a constant $A(f)$ depending on $f$ such that
\begin{equation}\label{assumptionSobolev}
\frac{\int_{\mathbb{R}} (1+|x|)^2f(x)^2 dx}
{\int_{\mathbb{R}} f(x)^2 dx}
\le A(f).
\end{equation}
Then, 
\begin{equation}\label{equationL2bound}
   \|f\|_{L^2}^2 \leq  (m_{R_1}\hat{m}_{R_2})^{-1} \left(\nstarn{f}^2   +  \frac{{m}_{R_1}}{n}\sum_{j=0}^{n-1}\|f \star \mu_{t_j}\|_{L^2([-R_1,R_1]^{\mathcal{C}})}^2 \right) +\|\widehat{f} \|^2_{L^2([-R_2,R_2]^{\mathcal{C}})} ,
\end{equation}
with
\begin{equation}
    m_{R_1} = g_1(2\pi\sigma_1^2)^{-1/2}\exp(-R_1^2(2\sigma_1^2)^{-1}), \hspace{0.5cm} \hat{m}_{R_2} = g_2 \exp\left(-2Ce^2\sigma_2^2R_2^2\log(30e^3R_2A(f))\right)/2,
\end{equation} 
where  $g_1, g_2,\sigma_1$ and $\sigma_2$ are defined in Lemma \ref{LemmaSubgaussianDistribution}.
\end{corollary}
The detailed proof of Corollary \ref{CorollaryL2Bound} can be found in  Section \ref{proof:CorollaryL2Bound}.

\noindent
We are now in a position to state the main result of the paper, an explicit $L^2(\R)-$error bound for the proposed estimator.
\begin{theorem}\label{theoremFinalErrorL2}
    Consider \cref{AssumptionProcessII}  and \cref{AssumptionInitialDistribution}. Assume that condition \eqref{conditionr1r2} holds %and  $\sigma_2^2 < 2\sigma_1^2$ for $\sigma_1,\sigma_2$ in Lemma \ref{LemmaSubgaussianDistribution}% 
    and $\eta := N^{-1/2}\vee \dn.$ For $\beta>0$  small enough, set 
    \begin{equation}
        A_\eta = \sqrt{2\beta\log(\eta^{-1})}, \qquad   2^{J_\eta} = O\left(\left(\eta^{-1}/{\log(\eta^{-1})}\right)^{1/(4r_2 + 6)}\right).
    \end{equation}
  Then, the estimator $\widehat{\varphi}_N$ defined in $\eqref{eqEstimator}$ over the class $S(A_\eta,J_\eta,C,r_1,r_2)$  satisfies 
      \begin{enumerate}
        \item[(i)] If $\dn\sqrt{N}\to0$, 
        \[
            \E\left[
            \|\varphi-\widehat{\varphi}_N\|_{L^2}^2
            \right]
            \lesssim
            \log(N)^{-r_1+\frac12}
            +
            \left(
            \frac{\log(N)}
            {\log\log(N)}
            \right)^{-r_2+\frac12}.
        \]

        \item[(ii)] If $\dn\sqrt{N}\to\infty$, 
        \[
            \E\left[
            \|\varphi-\widehat{\varphi}_N\|_{L^2}^2
            \right]
            \lesssim
            \log(\dn^{-1})^{-r_1+\frac12}
            +
            \left(
            \frac{\log(\dn^{-1})}
            {\log\log(\dn^{-1})}
            \right)^{-r_2+\frac12}.
        \]
    \end{enumerate}
 uniformly on the class $\varphi \in S(C,r_1,r_2)$.
\end{theorem}
\noindent
The detailed proof of Theorem \ref{theoremFinalErrorL2} can be consulted in Section \ref{prooftheoremFinalErrorL2}.

\begin{comment}
\begin{theorem}
    Consider \cref{AssumptionProcessII} and
    \cref{AssumptionInitialDistribution}. Assume that condition
    \eqref{conditionr1r2} holds and that
    $\sigma_2^2 \leq 2\sigma_1^2$ for $\sigma_1,\sigma_2$ in
    Lemma \ref{LemmaSubgaussianDistribution}. 
    and, for $\beta>0$ sufficiently small, set
    \[
        A_\eta = \sqrt{2\beta\log(\eta^{-1})},
    \]
    and
    \[
        2^{J_\eta}
        =
        O\left(
        \left(
        \frac{\eta^{-1}}
        {\sqrt{\log(\eta^{-1})}}
        \right)^{1/(4r_2+6)}
        \right).
    \]
    Then, the estimator $\widehat{\varphi}_N$ defined in
    \eqref{eqEstimator} over the class
    $S(A_\eta,J_\eta,C,r_1,r_2)$ satisfies
    \[
        \E\left[
        \|\varphi-\widehat{\varphi}_N\|_{L^2}^2
        \right]
        \lesssim
        \log(\eta^{-1})^{-r_1+\frac12}
        +
        \left(
        \frac{\log(\eta^{-1})}
        {\log\log(\eta^{-1})}
        \right)^{-r_2+\frac12},
    \]
    uniformly over the class
    $\varphi\in S(C,r_1,r_2)$.

    More precisely, the following two regimes hold:
\end{theorem}
\end{comment}

In the regime $\dn N^{1/2} \rightarrow 0,$ the discretization error is negligible and, up to $\log\log (N)$ factors, Theorem \ref{theoremFinalErrorL2} states that 
\begin{equation}\label{eq:finalRateClear}
            \E\left[
            \|\varphi-\widehat{\varphi}_N\|_{L^2}^2
            \right]
            \lesssim
            \log(N)^{-(r_1 \wedge r_2)+\frac{1}{2}}.
\end{equation}
Thus, the estimator converges in $L^2(\mathbb{R})$ at a logarithmic rate, with the rate determined by the smaller of the two regularity parameters $r_1$ and $r_2$, describing the spatial decay and the Fourier regularity of the interaction function, respectively.

It is worth noting that although the empirical error bound derived in Section~\ref{sec:MinimizationCompactFE} is not expected to be optimal, this has no direct impact on the final rate \eqref{eq:finalRateClear}. The intuition is that the exponential ill-posedness of the inversion of $\mu_t$ inevitably dominates the asymptotic behaviour, so that any polynomial improvement of the empirical error would still lead to the same logarithmic $L^2(\mathbb{R})$-convergence rate.

\subsection{Understanding the convergence rate}
The aim of this subsection is to provide some intuition for the convergence rates obtained in Theorem \ref{theoremFinalErrorL2} and to give the reader a better understanding of their origin. We begin by recalling the classical deconvolution problem. Consider the model
\[
Y_i=X_i+\varepsilon_i,\qquad i=1,\ldots,N,
\]
where $(X_i)_{i\ge1}$ and $(\varepsilon_i)_{i\ge1}$ are mutually independent i.i.d.\ sequences, with $X_i$ having an unknown density $f_X$ and the noise $\varepsilon_i$ having a known density $f_{\varepsilon}$. In the classical setting, the variables $Y_i$ are observed and the objective is to recover $f_X$. Passing to the Fourier domain and assuming that $\Phi_{\varepsilon}$ does not vanish, we obtain
\begin{equation}
    \Phi_X(z) = \Phi_Y(z)\Phi_{\varepsilon}^{-1}(z),
\end{equation}
where $\Phi_X$, $\Phi_Y$, and $\Phi_{\varepsilon}$ denote the Fourier transforms of $f_X$, $f_Y$, and $f_{\varepsilon}$, respectively. When the noise is Gaussian, the inverse Fourier multiplier $\Phi_{\varepsilon}^{-1}$ grows exponentially, giving rise to a severely ill-posed deconvolution problem. Such inverse problems are commonly referred to as supersmooth. It is well known that the corresponding minimax convergence rates over Sobolev and related smoothness classes are logarithmic rather than polynomial, with an exponent determined by the regularity of $f_X$; see, for instance, \cite{car88} and \cite[Chapter~2]{mei09}. In particular, if $f_X \in \mathcal{A}(K,r_1,r_2)$, then
\begin{equation}
    \E [\|f_X - \widehat{f}_X^R\|_{L^2}^2]
    \lesssim \log(N)^{-r_2+1/2},
\end{equation}
where $\widehat{f}_X^R$ denotes the classical cutoff estimator of the density $f_X$.

Comparing this result with the rate obtained in Theorem \ref{theoremFinalErrorL2}(i), we observe that, although the latter is also logarithmic, its exponent depends not only on $r_2$ but also on $r_1$. This indicates that, despite an apparent similarity with classical deconvolution, our statistical problem is of a fundamentally different nature. To illustrate this point, we introduce a simplified deconvolution problem that retains the main features responsible for the convergence rates in our interacting particle system while abstracting away from its technical complications. To the best of our knowledge, this simplified model has not previously been studied in the inverse problems literature.

Starting from our particle system, we make three simplifications. First, we assume that the volatility process
\[
\sigma_t^i(X_t^i):=(\varphi\star \mu_t^N)(X_t^i)
\]
is observed without error. Second, we replace the empirical measure $\mu_t^N$ by its mean-field limit $\mu_t$ and the interacting particles $X^i$ by the i.i.d.\ processes $\overline X^i$ introduced in \eqref{eqMcKeanVlasovSDE}. Finally, we restrict attention to a single fixed time $t_0\in[0,T]$. This leads to the simplified observation model
\[
f(X_i)=(\varphi\star\mu)(X_i),\qquad i=1,\ldots,N,
\]
where $X_i:=\overline X_{t_0}^i$ are i.i.d.\ random variables with common distribution $\mu:=\mu_{t_0}$, and the observed data are given by
\[
(X_i,f(X_i))_{1\leq i\leq N}.
\]
We next describe the main features of this simplified deconvolution problem and outline a natural estimation strategy. Our purpose is not to provide a complete mathematical analysis of the model, but rather to isolate the mechanisms underlying the convergence rates obtained in Theorem \ref{theoremFinalErrorL2}. Motivated by Lemma \ref{LemmaSubgaussianDistribution}, we assume that the probability measure $\mu$ is subgaussian and that its characteristic function $\widehat{\mu}$ satisfies
\begin{align}
\label{boundcha}
g_3\exp(-\sigma_3x^2)
\leq |\widehat{\mu}(x)|
\leq g_4\exp(-\sigma_4x^2),
\qquad x\in\R,
\end{align}
for some constants $g_3,g_4,\sigma_3,\sigma_4>0$. The Gaussian decay in the lower bound reflects the behaviour established in Section \ref{proof:CorollaryL2Bound} by means of Cartan's lemma.
Since
\(f=\varphi\star\mu\),
passing to the Fourier domain yields
\begin{equation}
\label{eq:ToyExampleFourierRepresentation}
\Phi_{\varphi}(z)
=
\Phi_f(z)\Phi_{\mu}^{-1}(z).
\end{equation}
Thus, in order to estimate $\Phi_{\varphi}$ and subsequently recover $\varphi$ by Fourier inversion, one needs to estimate both $\Phi_f$ and $\Phi_{\mu}$. The latter can be estimated directly by the empirical characteristic function
\[
\widehat{\Phi}_{\mu}(z)
:=
\frac{1}{N}\sum_{j=1}^N \exp(izX_j).
\]
The estimation of the unknown function $f$, however, is more delicate. Indeed, it can be viewed as an interpolation problem in which the function is observed exactly, but only at the random design points $X_1,\ldots,X_N$. Due to the subgaussianity of $\mu$, the observations $X_1,\ldots,X_N$ are, with high probability, essentially confined to an interval of the form $[-A,A]$, where
\[
A=A_N\sim\sqrt{\log(N)}.
\]
Consequently, the data contain essentially no information about $f$ outside this interval, and one can therefore only expect to recover accurately the truncated function
\[
f_A:=f1_{[-A,A]}.
\]
For $\varphi\in\mathcal{A}(K,r_1,r_2)$, using again the subgaussianity of $\mu$, we obtain
\begin{align}
    \|f-f_A\|_{L^2}^2
    &=
    \int_{[-A,A]^c} f^2(x)\,dx \notag\\
    &=
    \int_{[-A,A]^c}
    \left(\int_{\R}\varphi(x-y)\mu(dy)\right)^2dx
    \lesssim A^{-2r_1+1}.
\end{align}
By Plancherel's identity, it follows that
\[
\|\widehat{f}-\widehat{f_A}\|_{L^2}^2
\lesssim
\log(N)^{-r_1+1/2}.
\]
Hence, the first rate appearing in Theorem \ref{theoremFinalErrorL2}(i) can be interpreted as a consequence of the spatial truncation inherent in the interpolation problem for $f$. In particular, this contribution has no direct counterpart in the classical deconvolution problem described above.

The second rate appearing in Theorem \ref{theoremFinalErrorL2}(i) is instead closely related to the usual supersmooth deconvolution phenomenon. Indeed, the Fourier representation \eqref{eq:ToyExampleFourierRepresentation} shows that recovering $\varphi$ requires division by $\Phi_\mu$. In view of \eqref{boundcha}, the inverse multiplier $\Phi_\mu^{-1}$ grows exponentially, resulting in the same type of severe ill-posedness as in classical Gaussian deconvolution. Although the present toy model is more involved than the classical deconvolution problem because both $\Phi_f$ and $\Phi_\mu$ must be estimated from the data, we expect that, for a suitable estimator $\widehat{\varphi}$, this mechanism contributes an error of order
\[
\log(N)^{-r_2+1/2}.
\]
Combining the two sources of error therefore leads to a logarithmic convergence rate depending on both $r_1$ and $r_2$, in agreement with Theorem \ref{theoremFinalErrorL2}(i).

This toy model thus suggests that the two regularity parameters have distinct statistical interpretations: the parameter $r_1$ controls the error arising from the spatial truncation imposed by the random design, whereas $r_2$ governs the Fourier truncation error associated with the supersmooth inverse problem. We conjecture that the resulting convergence rate for this simplified model is minimax optimal. Since the convergence rate of our proposed estimator for the original interacting particle system essentially matches the rate suggested by the toy model, despite the substantially greater complexity of the original setting, we further conjecture that our estimator is minimax optimal over the considered class of functions, up to a $\log\log(N)$ factor. A rigorous derivation of the corresponding minimax upper and lower bounds is, however, beyond the scope of the present paper and is left for future research.

\section{Auxiliary results and proofs}\label{sec:AuxiliaryProofs}
In this section we provide detailed proofs and some auxiliary results that complement the main results of the article.

\subsection{Auxiliary results}

\begin{lemma}\label{LemmaBoundW1p}
    Under \cref{AssumptionsProcess}, for each $p\geq 2$, 
    \begin{equation}
        \E \left[\sup_{t\leq T}W_1^p(\mu_{t}^N, \mu_{t})\right]^{1/p} \leq C_{p,T,B,L_{\varphi}}N^{-1/(2p)},
    \end{equation}
    for a constant $C_{p,T,L_{\varphi},B}$ depending on $p,T,B$ and $L_{\varphi}$.
\end{lemma}
\begin{proof}
The proof relies on a coupling argument. Let
$(\overline{X}^i)_{i=1}^N$ be the solutions of the McKean--Vlasov equation defined in \eqref{eqMcKeanVlasovSDE}, driven by the same Brownian motions $(W_t^i)_{i=1}^N$ as the interacting particle system $(X_t^i)_{i=1}^N$
and starting from the same initial conditions $X_0^i$. Following the notation
for empirical measure introduced in \eqref{eq:empiricalMeasureIPS}, we define the empirical measure of the McKean--Vlasov system as 
\begin{equation}
\overline{\mu}_t^N:=\frac{1}{N}\sum_{i=1}^N\delta_{\overline{X}_t^i}.
\end{equation}
By the triangle inequality,  
\begin{equation}\label{eqLemmaPchaos1}
    \E \left[\sup_{t\leq T}W_1^p(\mu_{t}^N, \mu_{t})\right] \leq 2^{p-1}\E \left[\sup_{t\leq T}W_1^p(\overline{\mu}_{t}^N, \mu_{t})\right]  + 2^{p-1}\E \left[\sup_{t\leq T}W_1^p(\overline{\mu}_{t}^N, \mu_{t}^N)\right].
\end{equation}
From Theorem 1 in \cite{Fou15} (in our case $d=1$, $q = \infty$), we have that 
\begin{equation}\label{eqBoundIIDempirical}
    \E \left[\sup_{t\leq T}W_1^p(\overline{\mu}_{t}^N, \mu_{t})\right] \leq \E \left[\sup_{t\leq T}W_p^p(\overline{\mu}_{t}^N, \mu_{t})\right] \leq C_pN^{-1/2},
\end{equation}
for $C_p>0$ depending on $p.$  Next, for each $i=1,\dots,N,$ let us denote $\Delta_t^i := X_t^{i}-\overline{X}_t^i.$ By the definition of $W_1$, it is easy to check that 
\begin{equation}\label{eqboundterminoSegundo}
    \E \left[\sup_{t\leq T}W_1^p(\overline{\mu}_{t}^N, \mu_{t}^N)\right] \leq \frac{1}{N}\sum_{i=1}^N\E\left[\sup_{t\leq T}|\Delta_t^1|^p\right].
\end{equation}
For $p \geq 2$,
\begin{align}
    \E\left[\sup_{t\leq T}|\Delta_t^1|^p\right] &\leq C_{p,T} \int_0^T\E\left[\left|
b(X_s^{1}, \mu_s^N)-b(\overline{X}_s^1, \mu_s)\right|^p
\right]ds \\
&+ C_{p}\E\left[\sup_{t\leq T}\left|\int_0^t\left((\varphi \star \mu_s^N)(X_s^{1})-(\varphi \star \mu_s)(\overline X_s^1)\right)\,dW_s^1\right|^p\right].
\end{align}
By Hölder's inequality and the Lipschitzianity of the drift, we obtain that 
\begin{equation}\label{equnir1}
C_{p,T} \int_0^T\E\left[\left|
b(X_s^{1}, \mu_s^N)-b(\overline{X}_s^1, \mu_s)\right|^p
\right]ds \leq C_{p,T,B}\int_0^T\E\left[\sup_{r \leq s}|\Delta_r^1|^p + W_2^p(\mu_s^N,\mu_s) \right]ds.
\end{equation}
Applying the triangle inequality for the Wasserstein distance, together with bound \eqref{eqBoundIIDempirical} and the fact that 
$\E\left[ W_2^p(\mu_s^N, \overline{\mu}_s^N)\right] \leq \E\left[\sup_{r \leq s}|\Delta_r^1|^p\right]$, we obtain  
\begin{equation}
C_{p,T,B}\int_0^T\E\left[\sup_{r \leq s}|\Delta_r^1|^p + W_2^p(\mu_s^N,\mu_s) \right]ds \leq C_{p,T,B}\left(\int_0^T\E\left[\sup_{r \leq s}|\Delta_r^1|^p\right]ds + N^{-1/2} \right) .
\end{equation}
Since 
\begin{equation}
\left|(\varphi \star \mu_s^N)(X_s^{1})-(\varphi \star \mu_s)(\overline{X}_s^1)\right|\leq L_\varphi \left|X_s^1- \overline{X}_s^1\right| + L_\varphi W_1(\mu_s^N,\mu_s),
\end{equation}
Using  Burkholder--Davis--Gundy inequality
\begin{align}\label{equnir2}
&\E\left[\sup_{t\leq T}\left|\int_0^t\left((\varphi \star \mu_s^N)(X_s^{1}) -(\varphi \star \mu_s)(\overline{X}_s^1)\right)\,dW_s^1\right|^p\right]  \\
&\leq C_{p,T,L_{\varphi}}\int_0^T\E\left[
\sup_{r\leq s}|\Delta_r^1|^p\right]ds+C_{p,T,L_{\varphi}}\int_0^T\E\left[  W_1^p(\mu_s^N,{\mu}_s) \right]ds.
\end{align}
Again, we have that  
\begin{equation}\label{equnir3}
   \E\left[ W_1^p(\mu_s^N,{\mu}_s) \right] \leq 2^{p-1}\E\left[W_1^p(\mu_s^N,\overline{\mu}^N_s) \right] + 2^{p-1}\E\left[W_1^p(\mu_s,\overline{\mu}^N_s)\right].
\end{equation}
Applying Jensen's inequality and the interchangeability of the particles, 
\begin{equation}
    2^{p-1}\E\left[W_1^p(\mu_s^N,\overline{\mu}^N_s) \right]\leq \frac{C_p}{N}\sum_{j=1}^N\E\left[\left|X_s^j-\overline{X}_{s}^j\right|^p\right] \leq C_p\E\left[|\Delta_s^1|^p\right].
\end{equation}
Last term in \eqref{equnir3} can be bounded by expression  \eqref{eqBoundIIDempirical}. 
Combining the results, it leads 
\begin{equation}\label{equnirfinal}
     \E\left[\sup_{t\leq T}|\Delta_t^1|^p\right] \leq C_{p,T,B,L_{\varphi}}\int_0^T\E\left[
\sup_{r\leq s}|\Delta_r^1|^p\right]ds + C_{p,T,B,L_{\varphi}}N^{-1/2}.
\end{equation}
Finally, applying Gronwall's Lemma to \eqref{equnirfinal} provides a bound for \eqref{eqboundterminoSegundo}, which combined with \eqref{eqBoundIIDempirical} proves the inequality.
\end{proof}

\begin{proposition}\label{PropositionLpControlDifferenceNorms}
    Under \cref{AssumptionsProcess}, for all $p\geq2,$
    \begin{equation}
       \mathbb{E} \left[ \sup_{f \in S_N}\left|\nnatural{f-\varphi}^2 - \nstarn{f - \varphi}^2\right|^p \right]^{1/p} \lesssim  \frac{{K_{\varphi}L_{\varphi}}}{N^{1/(2p)}}.
    \end{equation}
\end{proposition}
\begin{proof}
From the definition of the norms, for any $g\in S_N(2K_{\varphi}, 2L_{\varphi})$, we have that 
\begin{align}
      \nnatural{g}^2 - \nstarn{g}^2 = \frac{1}{n}\sum_{j=0}^{n-1}\left(\int_{\R}(g \star \mu_{t_j}^N)^2(x)d\mu_{t_j}^N(x) - \int_{\R}(g\star \mu_{t_j})^2(x)d\mu_{t_j}(x)\right).
\end{align}
Adding and subtracting  $\int_{\R}(g\star \mu_{t_j})^2(x)d\mu_{t_j}^N(x)$, we can express the difference of the norms as  
\begin{equation}
     \nnatural{g}^2 - \nstarn{g}^2 = D_1(g) + D_{2}(g),
\end{equation}
with 
\begin{align}
    D_1(g) &:= \frac{1}{n}\sum_{j=0}^{n-1}\left( \int_{\R}(g\star \mu_{t_j})^2(x)d\mu_{t_j}^N(x) - \int_{\R}(g\star \mu_{t_j})^2(x)d\mu_{t_j}(x)\right)\\
    & = \frac{1}{n}\sum_{j=0}^{n-1} \int_{\R}(g\star \mu_{t_j})^2(x)\left(d\mu_{t_j}^N(x) - d\mu_{t_j}(x) \right),
\end{align}
\begin{align}
    D_2(g) &:=  \frac{1}{n}\sum_{j=0}^{n-1} \left(\int_{\R}(g \star \mu_{t_j}^N)^2(x)d\mu_{t_j}^N - \int_{\R}(g \star \mu_{t_j})^2(x)d\mu_{t_j}^N\right) \\
    &= \frac{1}{n}\sum_{j=0}^{n-1} \int_{\R}\left((g\star \mu_{t_j}^N)^2(x) - (g \star \mu_{t_j})^2(x)\right)d\mu_{t_j}^N(x).
\end{align}
Let us first control $D_1(g)$. Since the elements in  $S_N(2K_{\varphi},2L_{\varphi})$ are uniformly Lipschitz and bounded,
\begin{align}
    |(g\star \mu_{t_j})^2(x) - (g\star \mu_{t_j})^2(y)| &\leq |(g\star \mu_{t_j})(x) + (g\star \mu_{t_j})(y)||(g\star \mu_{t_j})(x) - (g\star \mu_{t_j})(y)| \\
    &\leq 8K_{\varphi}L_{\varphi}|x-y|.
\end{align}
It follows from the Kantorovich–Rubinstein duality that 
\begin{equation}\label{D1control}
    |D_1(g)| \leq \frac{8K_{\varphi}L_{\varphi}}{n}\sum_{j=0}^{n-1}W_1(\mu_{t_j}^N, \mu_{t_j}).
\end{equation}
On the other hand, we observe that 
\begin{align}
    |(g \star \mu_{t_j}^N)^2(x) - (g \star \mu_{t_j})^2(x)| &\leq 4K_{\varphi}\left|(g \star \mu_{t_j}^N)(x) - (g \star \mu_{t_j})(x)  \right|\\
    & = 4K_{\varphi}\left|\int_{\R}g(x-y)(d\mu_{t_j}^N(y) - d\mu_{t_j}(y))\right| \\
    &\leq 8K_{\varphi}L_{\varphi}W_1(\mu_{t_j}^N, \mu_{t_j}).
\end{align}
Then,
\begin{align}\label{D2control}
    |D_2(g)| &\leq \frac{1}{n}\sum_{j=0}^{n-1} \int_{\R} 8K_{\varphi}L_{\varphi}W_1(\mu_{t_j}^N, \mu_{t_j})d\mu_{t_j}^N
    \leq\frac{8K_{\varphi}L_{\varphi}}{n}\sum_{j=0}^{n-1}W_1(\mu_{t_j}^N, \mu_{t_j}).
\end{align}
Combining \eqref{D1control} and \eqref{D2control} yields
\begin{align}
   \mathbb{E} \left[ \sup_{f \in S_N}\left|\nnatural{f-\varphi}^2 - \nstarn{f - \varphi}^2\right|^p \right]^{1/p} 
   \leq 16{K_{\varphi}L_{\varphi}}\E \left[\sup_{t\leq T}W_1^p(\mu_{t}^N, \mu_{t})\right]^{1/p}.
\end{align}
Finally, applying Lemma \ref{LemmaBoundW1p} below, we conclude that 
\begin{equation}
    \E \left[\sup_{t\leq T}W_1^p(\mu_{t}^N, \mu_{t})\right]^{1/p} \leq C_{p,T,L_{\varphi}}N^{-1/(2p)},
\end{equation}
where $C_{p,T,L_{\varphi}}$ stands for a constant depending on $p,T$ and $L_{\varphi}$,  independently of $n,N$.  From the last bound, the   proposition follows.
\end{proof}

\begin{proposition}\label{PropositionLpControlGamma'}
    Under \cref{AssumptionsProcess},  for all $p\geq 2,$
    \begin{equation}
       \mathbb{E} \left[ \sup_{f \in S_N}\left|{\gamma}'_N(f)\right|^p \right]^{1/p}  \lesssim {p \dn^{1/2}(1+L_{\varphi}K_{\varphi}) \left( \text{DI}\left( {S}_N(2K_{\varphi},2L_{\varphi}), \|\cdot\|_{\infty}, \psi_{e,1} \right) + K_{\varphi}\right)}.
    \end{equation}
\end{proposition}
\begin{proof} To begin, we introduce the following random functions 
\begin{align}
    {E}'_{N,n}(f) &:= \frac{2}{nN}\sumsum (f \star {\mu}_{t_j}^{N})({X}_{t_j}^i)\left(C_{\dn}\left|{X}_{t_{j+1}}^i - {X}_{t_j}^i\right| -C_{\dn}\left|(\varphi \star {\mu}_{t_j}^N)({X}_{t_j}^i)\Delta W_{t_j}^i\right|\right),\\
    {M}'_{N,n}(f) &:= \frac{2}{nN}\sumsum (f \star {\mu}_{t_j}^{N})({X}_{t_j}^i) 
    \left(C_{\dn}\left|(\varphi \star {\mu}_{t_j}^{N})({X}_{t_j}^i)\Delta W_{t_j}^i\right|-(\varphi \star {\mu}_{t_j}^{N})({X}_{t_j}^i) \right).
\end{align}
with $C_{\dn} = \sqrt{{\pi}/({2\dn)}}$ and $\Delta W_{t_j}^i = W_{t_{j+1}}^i - W_{t_j}^i$. We note that $\gamma'_N(f) = {E}'_{N,n}(f) +  {M}'_{N,n}(f).$ 
First, we analyse the discretization error term, ${E}'_{N,n}(f)$.
For each $1\leq i \leq N$ and $ 0\leq j \leq n-1$, we define the random variable
\begin{equation}
    e_j^i(f) := (f \star {\mu}_{t_j}^{N})({X}_{t_j}^i)\left(C_{\dn}\left|{X}_{t_{j+1}}^i - {X}_{t_j}^i\right| -C_{\dn}\left|\left(\varphi \star {\mu}_{t_j}^N\right)({X}_{t_j}^i)\Delta W_{t_j}^i\right|\right).
\end{equation}
Since  the absolute value of the increments of the process can be bounded as 
\begin{equation}
    \left|{X}_{t_{j+1}}^i - {X}_{t_j}^i\right| \leq \left| \int_{t_j}^{t_{j+1}} b({X_s^i}, \mu_s^N)ds \right|  +  \left|\int_{t_j}^{t_{j+1}}(\varphi \star \mu_{s}^N)({X}_{s}^i)dW_s^i\right|,
\end{equation}  for all $p \geq 1$, using that $|x| - |y| \leq |x-y|, $
\begin{align}
  a^{-p}\E\left[|e_j^i(f)|^p \right] &\leq \E \left[\left| \int_{t_j}^{t_{j+1}} b({X_s^i}, \mu_s^N)ds \right|^p\right] \nonumber \\ \label{eqh2}
  &+ \E\left[\left|\int_{t_j}^{t_{j+1}} \left((\varphi \star \mu_{s}^N)({X}_{s}^i) - (\varphi \star {\mu}_{t_j}^N)({X}_{t_j}^i)\right)dW_s^i\right|^p\right],
\end{align}
with $a^p =  2^{p-1}C_{\dn}^{p} \|f  \|_{\infty}^{p}$. Since the drift is Lipschitz and $\E[W_2^p(\mu_s^N,\delta_0)] \leq \E[|X_s^1|^p]$,
\begin{equation}
    \E \left[\left| \int_{t_j}^{t_{j+1}} b({X_s^i}, \mu_s^N)ds \right|^p\right] \leq \dn^{p-1}B_{}^p\int_{t_j}^{t_{j+1}}(1 + \E[|{X}_{s}^i|^p] + \E\left[W_2^p(\mu_s^N,\delta_0)\right])ds \lesssim \dn^p.
\end{equation}
To control \eqref{eqh2},  we add and subtract $(\varphi \star {\mu}_s^N)({X}_{t_j}^i)$ and bound the term as
\begin{align}
    \E\left[\left|\int_{t_j}^{t_{j+1}} \left((\varphi \star \mu_{s}^N)({X}_{s}^i) - (\varphi \star {\mu}_{t_j}^{N})({X}_{t_j}^i)\right)dW_s^i\right|^p\right] \lesssim A^{i,j}_p + B^{i,j}_p,
\end{align}
with 
\begin{equation}
    A^{i,j}_p = \E\left[\left|\int_{t_j}^{t_{j+1}} \left((\varphi \star \mu_{s}^N)({X}_{s}^i) - (\varphi \star {\mu}_{s}^{N})({X}_{t_j}^i)\right)dW_s^i\right|^p\right], 
\end{equation}
and
\begin{equation}
    B^{i,j}_p = \E\left[\left|\int_{t_j}^{t_{j+1}} \left((\varphi \star {\mu}_{s}^{N})({X}_{t_j}^i) - (\varphi \star {\mu}_{t_j}^{N})({X}_{t_j}^i)\right)dW_s^i\right|^p\right]. 
\end{equation}
Using Burkholder-Davis-Gundy inequality and Lemma \ref{LemmaBoundingIncrements} below, we obtain that 
\begin{align}
    A_p^{i,j} &\leq C_p\E\left[\left( \int_{t_j}^{t_{j+1}}  \left((\varphi \star \mu_{s}^N)({X}_{s}^i) - (\varphi \star {\mu}_{s}^{N})({X}_{t_j}^i)\right)^2ds\right)^{p/2}\right]\\   &
    \lesssim L_{\varphi}^{p}p^{p/2}\E \left[\left(\int_{t_j}^{t_{j+1}}\left|{X}_s^i - {X}_{t_j}^i\right|^2ds\right)^{p/2}\right] \lesssim L_{\varphi}^{p}K_{\varphi}^p\dn^pp^{p}.
\end{align}
Similarly, 
\begin{align}
    B^{i,j}_p &\leq C_p\E \left[ \left(\int_{t_j}^{t_{j+1}}\left((\varphi \star {\mu}_{s}^{N})({X}_{t_j}^i) - (\varphi \star {\mu}_{t_j}^{N})({X}_{t_j}^i)\right)^2ds\right)^{p/2}\right]\\
    &\lesssim L_{\varphi}^pp^{p/2}\E \left[\left(\int_{t_j}^{t_{j+1}}W_1^2(\mu_s^N,\mu_{t_j}^N)ds\right)^{p/2}\right] \\
    &\lesssim L_{\varphi}^pp^{p/2}\E \left[\left(\int_{t_j}^{t_{j+1}}\left( \frac{1}{N}\sum_{k=1}^N\left|X_{s}^k - X_{t_j}^k\right|\right)^2ds\right)^{p/2}\right].
    %here Jensen 1 and then jensen to (1/N sum b_k)^p/2
\end{align}
Using that particles are exchangeable, we obtain that 
\begin{equation}
    B_p^{i,j}\lesssim {L_{\varphi}^pp^{p/2}} \E \left[\left(\int_{t_j}^{t_{j+1} }\left|X_{s}^k - X_{t_j}^k\right|^2ds\right)^{p/2}\right]\lesssim L_{\varphi}^pK_{\varphi}^p\dn^pp^{p}. 
\end{equation}
Combining the above estimates with \eqref{eqh2} and recalling the definition of $a$ above, we obtain that
\begin{equation}\label{ineqeij}
     \E\left[|e_j^i(f)|^p \right]^{1/p} \lesssim p(1 +L_{\varphi}K_{\varphi})\|f\|_{\infty}\dn^{1/2}.
\end{equation}
Since $e^i_j(f)$ is linear with respect to $f$, inequality \eqref{ineqeij} implies that for $f,g \in S_N(K_{\varphi},L_{\varphi})$, the random variable 
  $e_j^i(f-g)$ is subexponential, satisfying 
$ \| e_j^i(f-g) \|_{\psi_{e,1}} \lesssim (1 +L_{\varphi}K_{\varphi})\|f-g\|_{\infty}\dn^{1/2}.$ It implies that
\begin{equation}
    \|  {E}'_{N,n}(f)  -  {E}'_{N,n}(g) \|_{\psi_{e,1}} \lesssim (1 +L_{\varphi}K_{\varphi})\dn^{1/2}\|f-g\|_{\infty}.
\end{equation}
From Theorem \ref{TheoremMaximalInequality}, we obtain 
\begin{equation}
\left\|\sup_{f,g \in S_N} |{E}'_{N,n}(f)  -  {E}'_{N,n}(g)|\right\|_{\psi_{e,1}} \lesssim \dn^{1/2}(1 + L_{\varphi}K_{\varphi})\text{DI}\big(S_N(2K_{\varphi},2L_{\varphi}), \|\cdot\|_{\infty}, \psi_{e,1}\big).
\end{equation}
On the other hand, for any fixed $h_0 \in S_N(K_{\varphi},L_{\varphi})$,
$ \|{E}'_{N,n}(h_0)  \|_{\psi_{e,1}} \lesssim \dn^{1/2} (1 +L_{\varphi}K_{\varphi})K_{\varphi}.$
Therefore,
\begin{equation}
  \left\|  \sup_{f \in S_N} |{E}'_{N,n}(f)|\right\|_{\psi_{e,1}} \lesssim \dn^{1/2}(1+L_{\varphi}K_{\varphi})(\text{DI}\big(S_N(2K_{\varphi},2L_{\varphi}), \|\cdot\|_{\infty}, \psi_{e,1}\big) + K_{\varphi}).
\end{equation}
Since for any random variable $\|X\|_{L^p} \lesssim p \left\| X \right\|_{\psi_{e,1}} $, 
we conclude that 
\begin{equation}\label{eqBoundENn}
    \E \left[   \sup_{f \in S_N} |{E}'_{N,n}(f)|^p\right]^{1/p} \lesssim p\dn^{1/2}(1+L_{\varphi}K_{\varphi})(\text{DI}\big(S_N(2K_{\varphi},2L_{\varphi}), \|\cdot\|_{\infty}, \psi_{e,1}\big) + K_{\varphi}).
\end{equation}
Next, we study the function ${M}'_{N,n}(f)$.\\
For $i=1,\dots,N, \hspace{0.1cm} 0\leq j\leq  n-1$, we introduce the random variables $\epsilon_j^i = C_{\dn}|\Delta W_{t_j}^i| - 1$. It is easy to check that  $\epsilon_j^i$ are subgaussian and independent. In particular, for $\lambda \in \R$, there exists $c>0$ independent of $\dn$ such that $\E\left[\exp(\lambda\epsilon_j^i)\right] \leq \exp(c\lambda^2).$ Let us consider the filtration
$\mathcal{G}_{t_k}^N = \sigma\left(X_{t_s}^i \hspace{0.05cm}| \hspace{0.05cm}0\leq s\leq k; \hspace{0.1cm}i=1,\dots,N\right)$ and define the random functional
\begin{equation}
    m_j(f) := \frac{2}{nN}\sum_{i=1}^N(f \star {\mu}_{t_j}^{N})({X}_{t_j}^i) (\varphi \star {\mu}_{t_j}^{N})({X}_{t_j}^i)
    \epsilon_j^i.
\end{equation}
We note that for a fixed $j$,  $(f \star {\mu}_{t_j}^{N})({X}_{t_j}^i) (\varphi \star {\mu}_{t_j}^{N})({X}_{t_j}^i)$ is measurable with respect to $\mathcal{G}_{t_j}^N$ and $\epsilon_j^i \perp \mathcal{G}_{t_j}^N$. Then, the conditional moment generator function can be bounded as 
\begin{equation}
    \E\left[\exp(\lambda m_j(f)) | \mathcal{G}_{t_j}^N\right] \leq \exp\left(\frac{c\lambda^2\|f\|_{\infty}^2K_{\varphi}^2}{n^2N}\right).
\end{equation}
Since the volatility of \eqref{eqIPS} is assumed non-negative, we note that $M'_{N,n}(f)$ can be expressed as ${M}'_{N,n}(f) =\sum_{j=0}^{n-1} m_j(f).$ 
We observe that
\begin{align}
    \E\left[\exp\left(\lambda M'_{N,n}(f)\right)\right] &= \E\left[\prod_{j=0}^{n-2}\exp\left(\lambda m_j(f)\right)\E\left[\exp(\lambda m_{n-1}(f)) | \mathcal{G}_{t_{n-1}}^N\right]\right]\\
    &\leq \exp\left(\frac{c\lambda^2\|f\|_{\infty}^2K_{\varphi}^2}{n^2N}\right)\E\left[\prod_{j=0}^{n-1}\exp\left(\lambda m_j(f)\right)\right].
\end{align}
Applying the same argument recursively, we conclude that 
\begin{equation}\label{eqimpliesSubgaussian}
    \E\left[\exp(\lambda M'_{N,n}(f))\right] \leq \exp\left(\frac{c\lambda^2\|f\|_{\infty}^2K_{\varphi}^2}{nN}\right).
\end{equation}
Since $M'_{N,n}(f)$ is linear with respect to $f$, equation \eqref{eqimpliesSubgaussian} implies that for any $f,g \in S_N(K_{\varphi},L_{\varphi}),$
\begin{equation}
\|M'_{N,n}(f)-M'_{N,n}(g)\|_{\psi_{e,2}} \lesssim \frac{K_\varphi}{\sqrt{nN}} \|f-g\|_\infty.
\end{equation}
Again, by Theorem \ref{TheoremMaximalInequality},
\begin{equation}
\left\|\sup_{f,g\in S_N}|M'_{N,n}(f)-M'_{N,n}(g)|\right\|_{\psi_{e,2}}\lesssim \frac{K_\varphi}{\sqrt{nN}}\text{DI}\left(S_N(2K_{\varphi},2L_{\varphi}),\|\cdot\|_\infty,\psi_{e,2}\right),
\end{equation}
and for any fixed \(h_0\in S_N(K_\varphi,L_{\varphi})\), $\|M'_{N,n}(h_0)\|_{\psi_{e,2}} \lesssim {K_\varphi^2}({nN})^{-1/2}.$ Hence,
\begin{equation}
\left\|\sup_{f\in S_N} |M'_{N,n}(f)| \right\|_{\psi_{e,2}} \lesssim \frac{K_\varphi}{\sqrt{nN}} \left(\text{DI}\left(S_N(2K_{\varphi},2L_{\varphi}),\|\cdot\|_\infty,\psi_{e,2}\right) + K_\varphi \right),
\end{equation}
which implies 
\begin{equation}\label{eqboundMNn}
\E\left[ \sup_{f\in S_N} |M'_{N,n}(f)|^p \right]^{1/p} \lesssim \sqrt{p} \frac{K_\varphi}{\sqrt{nN}} \left( \text{DI}\left(S_N(2K_{\varphi},2L_{\varphi}),\|\cdot\|_\infty,\psi_{e,2}\right) + K_\varphi \right).
\end{equation}
Once we have uniform control over $M'_{N,n}(f)$ and ${E}'_{N,n}(f)$, 
\begin{align}
     \mathbb{E} \left[ \sup_{f \in S_N}\left|{\gamma}'_N(f)\right|^p \right]^{1/p} &\lesssim \left( \E\left[\sup_{f \in S_N}\left|{E}'_{N,n}(f)\right|^p  \right] + \E\left[\sup_{f \in S_N}\left|{M}'_{N,n}(f)\right|^p  \right]\right)^{1/p} 
\end{align}
Comparing \eqref{eqBoundENn} and \eqref{eqboundMNn}, since $\dn =T/n$ for a fixed $T >0$ and $N \rightarrow \infty$, we  conclude that the martingale term $M'_{N,n}(f)$ is  negligible compared to the discretization term ${E}'_{N,n}(f)$. Therefore 
\begin{equation}
   \mathbb{E} \left[ \sup_{f \in S_N}\left|{\gamma}'_N(f)\right|^p \right]^{1/p} \lesssim\left( 2\E\left[\sup_{f \in S_N}\left|{E}'_{N,n}(f)\right|^p  \right] \right)^{1/p}.
\end{equation}
and bound \eqref{eqBoundENn} completes the proof.
 \end{proof}

\begin{lemma}\label{LemmaBoundingIncrements} Under \cref{AssumptionsProcess}, for all $p\geq 2$, 
\begin{equation}
    \E\left[\left( \int_{t_j}^{t_{j+1}} |{X}_s^i - {X}_{t_j}^i|^2 ds\right)^{p/2} \right]^{1/p} \lesssim \sqrt{p}K_{\varphi}\dn.
\end{equation}
\end{lemma}
\begin{proof}
By Jensen inequality,
\begin{align}
    \E\left[\left( \int_{t_j}^{t_{j+1}} |{X}_s^i - {X}_{t_j}^i|^2 ds\right)^{p/2} \right] \leq \dn^{p/2-1}\int_{t_j}^{t_{j+1}}\E \left[|{X}_s^i - {X}_{t_j}^i|^p \right] ds.
\end{align}
By definition,  
\begin{equation}
    \E\left[|{X}_s^i - {X}_{t_j}^i|^p\right] \leq 2^{p-1}\left(\E\left[\left|\int_{t_j}^s b({X}_u^i,\mu_u^N)du \right|^p\right] + \E\left[\left| \int_{t_j}^s (\varphi \star \mu_{u}^N)({X}_u^i)dW_u^i \right|^p\right]\right).
\end{equation}
Using $\mathcal{A}_1)$ and \eqref{eq:MomentsBoundforallP}, we obtain that 
\begin{equation}
    \E\left[\left|\int_{t_j}^s b({X}_u^i,\mu_u^N)du \right|^p\right] \leq \dn^{p-1}B^{p} \int_{t_j}^s\E[(1 + |{X}_u^i| + W_2(\mu_u^N,\delta_0))^p]du \lesssim \dn^p.
\end{equation}
On the other hand, using BDG inequality and $\mathcal{A}_2),$
\begin{equation}
    \E\left[\left| \int_{t_j}^s (\varphi \star \mu_{u}^N)({X}_u^i)dW_u^i \right|^p\right] \lesssim p^{p/2} \E\left[\left(\int_{t_j}^s  (\varphi \star \mu_{u}^N)^2(X_u^i)du\right)^{p/2} \right] \lesssim K_{\varphi}^p\dn^{p/2}p^{p/2}.
\end{equation}
Combining both results leads
\begin{align}
    \E\left[\left( \int_{t_j}^{t_{j+1}} |{X}_s^i - {X}_{t_j}^i|^2 ds\right)^{p/2} \right] \lesssim  \dn^{p/2-1}\int_{t_j}^{t_{j+1}}(\dn^p + K_{\varphi}^p\dn^{p/2}p^{p/2}) ds \lesssim K_{\varphi}^p\dn^p p^{p/2}.
\end{align}
\end{proof}
\noindent
The following technical lemma will be used in the proof of Theorem \ref{theoremFinalErrorL2}.
\begin{lemma}\label{Lemmalogcontrol}
Let $(a_n)_{n\geq 1}$ be a sequence of positive real numbers such that
\begin{equation}\label{eqlemma_assumption}
    a_n \leq n^{-w} + {|\log(a_n)|^q}{\log(n)^{-q}},
\end{equation}
for some constants $w,q>0$. Then there exists $n_0 \in \N $ such that for all $n\geq n_0$,
\begin{equation}\label{eqlemmaConclusion}
    a_n \lesssim \left(\frac{\log(n)  }{\log\log(n) }\right)^{-q}.
\end{equation}
\end{lemma}
\begin{comment}
\begin{proof}
First, we observe that, since $n \rightarrow\infty,$ the sequence $a_n$ must go to 0 in order to satisfy \eqref{eqlemma_assumption} for all $n \in \N$. Set $r>1$ and define 
\begin{equation}
    s_n = r\left(\frac{\log\log(n)}{\log(n)}\right)^q,
\end{equation}
Since 
\begin{equation}
    |\log(s_n)| \lesssim q\log\log(n) \hspace{0.5cm} \text{and}\hspace{0.5cm} n^{-w} = o\left(\left(\frac{\log(n)}{\log \log(n)}\right)^{-q}\right),
\end{equation}
for a sufficiently large $r$, there exists $n_0 \in \N$ such that $\forall n \geq n_0$ we have that  $a_n \leq s_n$. In particular, \eqref{eqlemmaConclusion} holds.
\end{proof}
\end{comment}
\noindent
A detailed proof of this lemma can be found in  \cite[Lemma 5.1]{bel24}.

\subsection{Remaining Proofs}
\subsubsection{Proof of Proposition \ref{PropositionInfimumApproximationError}}\label{proof:PropApproximationError}
\begin{proof}
Since $\varphi \in S(C,r_1,r_2)$, it follows that 
\begin{equation}\label{eqineq1}
    \inf_{f\in S(A,J,C,r_1,r_2)} \nstarn{f-\varphi}^2 \lesssim \nstarn{\sum_{j\geq J}\sum_{k\in \Z}c_{jk}\psi_{jk}}^2 + \nstarn{\sum_{j \geq-1}^{J-1}\sum_{\substack{k\in\Z\\ \operatorname{supp}(\psi_{jk})\not\subseteq [-A,A]}}c_{jk}\psi_{jk}}^2.
\end{equation}
We first estimate the first term on the right-hand side of \eqref{eqineq1}.
We observe that for all $j \geq J, k\in \Z$, since $\psi$ is compactly supported, 
\begin{equation}
    |(\psi_{jk} \star \mu_{t_l})(x)| \leq \int_{\R}2^{j/2}|\psi_{}(2^{j}(x-y) - k)\mu_{t_l}(y)|dy \lesssim 2^{-j/2}.
\end{equation}
Since $r_1>1$, we obtain 
\begin{align}\label{eqresult1}
    \nstarn{\sum_{j\geq J}\sum_{k\in \Z}c_{jk}\psi_{jk}}^2 \leq \frac{1}{n}\sum_{l=0}^{n-1}\int_{\R}\left(C\sum_{j \geq J} 2^{-j(r_2 + 1)}\sum_{k \in \Z}\frac{1}{(1 + |k|)^{r_1}}\right)^2\mu_{t_l}(x)dx \lesssim 2^{-2J(r_2 +1)}.
\end{align}
Next, let us denote
\begin{equation}
    g(x) := \sum_{j \geq-1}^{J-1}\sum_{\substack{k\in\Z\\ \operatorname{supp}(\psi_{jk})\not\subseteq [-A,A]}}c_{jk}\psi_{jk}(x).
\end{equation}
By construction,  supp$(g) \subset [-A+c_0,A-c_0]^{\mathcal{C}}$ for a fixed $c_0>0$. Applying the definition of the seminorm,  
\begin{equation}
    \nstarn{g}^2 = \frac{1}{n}\sum_{l=0}^{n-1}\int_{-A/2+c_0}^{A/2 - c_0}(g \star \mu_{t_l})^2(x)\mu_{t_l}(x)dx + \frac{1}{n}\sum_{l=0}^{n-1}\int_{[-A/2+c_0,A/2-c_0]^{\mathcal{C}}}(g \star \mu_{t_l})^2(x)\mu_{t_l}(x)dx.
\end{equation}
 First, from  Lemma \ref{LemmaSubgaussianDistribution}, we have that 
for any $x\in [-A/2+c_0, A/2-c_0]$,
\begin{align}
   |(g \star\mu_{t_l})(x)| \leq \int_{\R}|g(y)|\mu_{t_l}(x-y)dy \leq \|g\|_{\infty}g_2\int_{|z| > A/2}\phi_{\sigma_2}(z)dz.
\end{align}
Using the change of variables $t= z^2/(2\sigma_2^2)$ and  the approximation  $\Gamma(s,t) \sim t^{s-1}\exp(-t)$ for $t\rightarrow \infty$, where 
\begin{equation}
    \Gamma(s,t) = \int_{t}^\infty u^{s-1}\exp(-u)du,
\end{equation}
we obtain that for $x\in [-A/2+c_0, A/2-c_0]$, 
\begin{equation}
     |(g \star\mu_{t_l})(x)| \lesssim \frac{\|g\|_{\infty}}{A^{}}\exp\left(-\frac{A^2}{8\sigma_2^2}\right).
\end{equation}
Therefore, 
\begin{align}\label{eqineq2}
    \frac{1}{n}\sum_{l=0}^{n-1} \int_{-A/2+c_0}^{A/2-c_0}(g \star \mu_{t_l})^2(x)\mu_{t_l}(x)dx \lesssim K^2A^{-2}\exp(-A^2/(4\sigma_2^2)).
\end{align}
On the other hand, since $\|g\|_{\infty} \leq K$ and using Lemma $\ref{LemmaSubgaussianDistribution}$ below,
\begin{align}
     \frac{1}{n}\sum_{l=0}^{n-1} \int_{[-A/2+c_0,A/2-c_0]^{\mathcal{C}}}((g \star \mu_{t_l})(x))^2\mu_{t_l}(x)dx &\lesssim K^2g_2 \int_{|x|> A/2} \phi_{\sigma_2}(x)dx \nonumber\\ \label{eqineq3}
     &\lesssim K^2 A^{-1}\exp(-A^2/(8\sigma_2^2)).
\end{align}
Combining \eqref{eqineq2} and \eqref{eqineq3}, for $A$ sufficiently large, we obtain
\begin{equation}\label{eqresult2}
    \nstarn{g}^2 \lesssim K^2 A^{-1}\exp(-A^2/(8\sigma_2^2)).
\end{equation}
From \eqref{eqresult1} and \eqref{eqresult2}, we prove the desired bound.
\end{proof}

\subsubsection{Proof of Proposition \ref{PropositionRelationL2NStarn}}\label{proof:PropositionRelationL2NStarn}
\begin{proof}
    From the definition of the $\nstarn{.}$ norm, for $R_1>0,$ we have that 
    \begin{align}
        \nstarn{f}^2 &= \frac{1}{n}\sum_{j=0}^{n-1}\int_{\R} (f \star \mu_{t_j})(x)^2\mu_{t_j}(x)dx \geq \frac{1}{n}\sum_{j=0}^{n-1}\int_{-R_1}^{R_1} (f \star \mu_{t_j})^2(x)\mu_{t_j}(x)dx \\
        &\geq \frac{1}{n}\sum_{j=0}^{n-1} m_{R_1}^{t_j}\|f \star \mu_{t_j}\|_{L^2([-R_1,R_1])}^2 =  \frac{1}{n}\sum_{j=0}^{n-1} m_{R_1}^{t_j} \left(\|f \star \mu_{t_j}\|_{L^2}^2 - \|f \star \mu_{t_j}\|_{L^2([-R_1,R_1]^{\mathcal{C}})}^2\right).
    \end{align}
    For each $t_j$, applying Plancherel’s  identity, 
    \begin{align}\label{s1}
       \|f \star \mu_{t_j}\|_{L^2}^2 &= \int_{\R} |\widehat{f}(x)|^2 |\widehat{\mu}_{t_j}(x)|^2dx \geq \widehat{
       m}^{t_j}_{R_2} \left( \int_{\R} |\widehat{f}(x)|^2 dx - \int_{[-R_2,R_2]^{\mathcal{C}}} |\widehat{f}(x)|^2 dx \right)\\
       &\geq\widehat{
       m}^{t_j}_{R_2} \left( \|f\|_{L^2}^2 - \|\widehat{f} \|^2_{L^2([-R_2,R_2]^{\mathcal{C}})}   \right).
    \end{align}
    Therefore, 
    \begin{align}
        \nstarn{f}^2 \geq \frac{1}{n}\sum_{j=0}^{n-1}m_{R_1}^{t_j}\widehat{
       m}^{t_j}_{R_2}\|f\|_{L^2}^2  - \frac{1}{n}\sum_{j=0}^{n-1}m_{R_1}^{t_j}\widehat{
       m}^{t_j}_{R_2}\|\widehat{f} \|^2_{L^2([-R_2,R_2]^{\mathcal{C}})} \\-\frac{1}{n}\sum_{j=0}^{n-1}m_{R_1}^{t_j}\|f \star \mu_{t_j}\|_{L^2([-R_1,R_1]^{\mathcal{C}})}^2.
    \end{align}
    Rearranging terms completes the proof.
\end{proof}

\subsubsection{Proof of Corollary \ref{CorollaryL2Bound}.}\label{proof:CorollaryL2Bound}

\begin{proof}
    First, we point out that under \cref{AssumptionProcessII} and \cref{AssumptionInitialDistribution}, Lemma \ref{LemmaSubgaussianDistribution} holds. Therefore, for all $t_j$, we have that 
    \begin{equation}
        \inf_{x\in [-R_1,R_1]}\mu_{t_j}(x) \geq g_1(2\pi\sigma_1^2)^{-1/2}\exp(-R_1^2(2\sigma_1^2)^{-1}). 
    \end{equation}
We next study the analytic continuation of the Fourier
transform of $\mu_{t_j}$ to the complex plane. For $z =u + iv \in \C, $ we have that 
    \begin{align}
        |\hat{\mu}_{t_j}(z)| &= \left| \int_{\R} \exp(-izx)\mu_{t_j}(x)dx\right| \leq g_2\int_{\R}\exp(vx)\phi_{\sigma_2}(x)dx 
        \leq g_2\exp(\sigma_2^2|z|^2/2).
    \end{align}
    Hence, we can extend $\hat{\mu}_{t_j}$ to an entire function on $\C$ (by dominated convergence on compact subsets of $\mathbb{C}$). Moreover, its maximum modulus $M_{\hat{\mu}_{t_j}}(r):= \sup_{\{|z| = r\}} |\hat{\mu}_{t_j}(z)|$ satisfies 
    \begin{equation}\label{q1}
        \log(M_{\hat{\mu}_{t_j}}(r)) \leq \log(g_2) + \frac{\sigma_2^2}{2}r^2.
    \end{equation}
    %In particular, the order of $\widehat{\mu}_{t_j}$ satisfies $\rho(\widehat{\mu}_{t_j}) = 2$ and its type $\tau(\widehat{\mu}_{t_j}) \leq \sigma_2^2/2.$
    By a Cartan-type estimate (Theorem 4 Section 11.3 in \cite{Lev96}), 
%The Theorem:
\begin{comment}
Let $f$ be analytic in the disk
\[
B_{2eR}^{\mathbb C}=\{z\in\mathbb C:\ |z|<2eR\},
\]
and assume that $|f(0)|=1$. Then, for every $\eta\in(0,1)$,
\[
\log|f(z)|
\ge
-H(\eta)\log M_f(2eR),
\]
for all
\[
z\in B_R^{\mathbb C}\setminus E_{\eta,R}^{\mathbb C},
\]
where
\[
H(\eta)=\log\left(\frac{15e^3}{\eta}\right),
\]
and $E_{\eta,R}^{\mathbb C}$ is a union of disks of radii $r_i$
such that
\[
\sum_i r_i\le \eta R.
\]
\end{comment}  
    given $R_2>0,$ there exists a constant $C>0$ such that for every $\eta \in (0,1)$ we have that 
    \begin{equation}\label{q2}
        \log(|\hat{\mu}_{t_j}(z)|) \geq -C H(\eta) \log(M_{\hat{\mu}_{t_j}}({2eR_2})) \hspace{0.3cm}\forall z \in B_{R_2}^{\C} \backslash E_{\eta, R_2}^{\C}
    \end{equation}
    where $E_{\eta, R_2}^{\C}$ is a union of complex disks of radius $r_i$ inside $B_{R_2}^{\C}$ satisfying $\sum_{i}r_i \leq \eta R_2$ and $H(\eta) = \log(15e^3 \eta^{-1})$. Combining \eqref{q1} and \eqref{q2}, on $B_{R_2}^{\C} \backslash E_{\eta, R_2}^{\C}$, we have that 
    \begin{equation}\label{q3}
        |\hat{\mu}_{t_j}(z)| \geq \exp\left(-C H(\eta)e^2\sigma_2^2R_2^2\right).
    \end{equation}
    By a standard projection argument, we define $E_{\eta,R_2} =  E_{\eta, R_2}^{\C} \cap   \R$ satisfying $|E_{\eta,R_2}| \leq 2\eta R_2$. Then,  \eqref{q3} holds in $[-R_2,R_2] \backslash E_{\eta, R_2}$. Consequently, 
    \begin{align}
        \int_{-R_2}^{R_2} |\hat{f}(z)|^2 |\hat{\mu}_{t_j}(z)|^2 dz &\geq \int_{[-R_2,R_2] \backslash E_{\eta, R_2}} |\hat{f}(z)|^2 |\hat{\mu}_{t_j}(z)|^2 dz \nonumber \\
        &\geq \exp\left(-2C H(\eta)e^2\sigma_2^2R_2^2\right)\int_{[-R_2,R_2] \backslash E_{\eta, R_2}}|\hat{f}(z)|^2 dz. \label{w1}
    \end{align}
    Defining $r_{\eta,R_2}:= (\int_{E_{\eta,R_2}}|\hat{f}(z)|^2 dz)/(\int_{\R}|\hat{f}(z)|^2 dz)$, we  can express
    \begin{align}
        \int_{[-R_2,R_2] \backslash E_{\eta, R_2}}|\hat{f}(z)|^2 dz &= \left( \int_{\R \backslash E_{\eta, R_2}}|\hat{f}(z)|^2 dz - \int_{[-R_2,R_2]^{\mathcal{C}}} |\hat{f}(z)|^2 dz\right) \nonumber\\
        &=(1 - r_{\eta,R_2})\int_{\R} |\hat{f}(z)|^2 dz - \int_{[-R_2,R_2]^{\mathcal{C}}} |\hat{f}(z)|^2 dz \label{w2}.
    \end{align}
    Let us  control $r_{\eta,R_2}$. By the Sobolev embedding, 
    \begin{align}
        \int_{E_{\eta,R_2}}|\hat{f}(z)|^2 dz \leq |E_{\eta,R_2}|\|\hat{f}\|_{\infty}^2 \leq 2\eta R_2 Q_1\|\hat{f}\|_{H^1}^2 
    \end{align}
    where 
    \begin{equation}
        Q_1 := \frac{1}{2\pi}\int_{\R} (1 + |z|^2)^{-1}dz = \frac{1}{2}, \hspace{0.5cm} \|\hat{f}\|_{H^1}^2 := 2\pi \int_{\R}(1 + |z|^2)|f(z)|^2dz.
    \end{equation}
    Therefore, since \eqref{assumptionSobolev} holds and choosing $\eta = (2R_2A(f))^{-1}$, we obtain that 
    \begin{equation}\label{w3}
        r_{\eta,R_2} \leq \frac{\eta R_2  \|\hat{f}\|_{H^1}^2}{\int_{\R}|\hat{f}(z)|^2 dz} \leq 1/2.
    \end{equation}
    Combining \eqref{w1}, \eqref{w2} and \eqref{w3}, we have 
    \begin{equation}
         \int_{-R_2}^{R_2} |\hat{f}(z)|^2 |\hat{\mu}_{t_j}(z)|^2 dz > \hat{m}_{R_2}\left(  \int_{\R} |\hat{f}(z)|^2dz - \int_{[-R_2,R_2]^{\mathcal{C}}}  |\hat{f}(z)|^2 dz\right)
    \end{equation}
    with $\hat{m}_{R_2} =g_2 \exp\left(-2Ce^2\sigma_2^2R_2^2\log(30e^3R_2A(f))\right)/2.$
Then, following Proposition \ref{PropositionRelationL2NStarn}, we obtain  
\begin{equation}\label{ww}
       \nstarn{f}^2 \geq m_{R_1}\hat{m}_{R_2}\|f\|_{L^2}^2  - m_{R_1}\hat{
       m}_{R_2}\|\widehat{f} \|_{L^2([-R_2,R_2]^{\mathcal{C}})}^2 - m_{R_1}\frac{1}{n}\sum_{j=0}^{n-1}\|f \star \mu_{t_j}\|_{L^2([-R_1,R_1]^{\mathcal{C}})}^2.
\end{equation}
From the last inequality,  we obtain the desired result.
\end{proof}

\subsubsection{Proof of Theorem \ref{theoremFinalErrorL2}.}\label{prooftheoremFinalErrorL2}
\begin{proof}
 Let us consider $f\in S(C,r_1,r_2)$ and $f_A \in S(A_\eta,J_\eta,C,r_1,r_2)$ denote the truncation of $f$ to this subspace. From Lemma \ref{lemmaInclusionApproximationSpaces}, we note that both $f,f_A \in \mathcal{A}(K,r_1,r_2)$ for a value $K>0$.   
By the triangle inequality, we have that 
\begin{equation}\label{eqineqL2}
    \|f\|_{L^2}^2 \lesssim \|f-f_A\|_{L^2}^2 + \|f_A\|_{L^2}^2.
\end{equation}
Since supp$(f-f_A) \subset  [-A_\eta + c_0,A_\eta-c_0]^{\mathcal{C}}$ for a fixed $c_0 > 0$ and the tails of $f-f_A$ have polynomial decay since $f-f_A \in \mathcal{A}(K,r_1,r_2)$, we have that 
\begin{equation}\label{eqb00}
    \|f-f_A\|_{L^2}^2 \lesssim 2K^2 \int_{|x| > A_\eta-c_0}|x|^{-2r_1}dx \lesssim K^2A_\eta^{-2r_1 +1}.
\end{equation}
Next, we proceed to control the second term of the right-hand side of \eqref{eqineqL2} by using  Corollary \ref{CorollaryL2Bound}. 
For $b_1, b_2 > 0 $,  we define
\begin{equation}
    R_{1,\eta} : = \sqrt{2b_1\log(\eta^{-1})}, \hspace{1cm} R_{2, \eta}:=\sqrt{\frac{2b_2\log(\eta^{-1})}{\log\log(\eta^{-1}) + |\log(\|f_A\|_{L^2}^2)|}},
\end{equation}
and $A_{\eta}= \sqrt{2\beta \log(\eta^{-1})}$ for  $\beta>0.$ 
To control $\|f_A\|_{L^2}^2$ using Corollary \ref{CorollaryL2Bound}, we need to study the order of the terms in the right-hand side of  inequality \eqref{equationL2bound}. First, since $\widehat{f}_A$ has polynomial decay,
\begin{equation}\label{eqb0}
    \int_{|x| > R_{2,\eta}} \widehat{f}_A(x)^2dx \leq K^2  \int_{|x| > R_{2,\eta}} |x|^{-r_2}dx \lesssim K^2 R_{2,\eta}^{-2r_2 + 1}.
\end{equation}
Next, let us fix $r \in (0,1)$. Since $A_\eta = \sqrt{2\beta \log(\eta^{-1})}$, for $\beta$ small enough, we can ensure that $A_\eta < rR_{1,\eta}$. Therefore, 
\begin{align}
    \int_{|x| > R_{1,\eta}}(f_A \star \mu_{t_j})(x)^2dx &\leq \int_{|y| \leq rR_{1,\eta}} f_A(y)^2 \int_{|x| > R_{1,\eta}} \mu_{t_j}(x-y)dxdy \\
    &\lesssim K^2 R_{1,\eta} \int_{|z| > (1 -r)R_{1,\eta}} \mu_{t_j}(z)dz
    \\
    &\lesssim K^2R_{1,\eta}\phi_{\sigma_2}((1-r)R_{1,\eta}) \lesssim_{\log(\eta^{-1})} \eta^{\frac{(1 - r)^2b_1}{\sigma_2^2}}\label{eqb1},
\end{align}
where $\alpha_1\lesssim_{\log(\eta^{-1})} \alpha_2$ means that $\alpha_1 \leq C(\log(\eta^{-1}))^q \alpha_2$ for $C,q>0$ independent of $\eta$.
Again, since $f_A \in S(A_\eta,J_\eta,C,r_1,r_2)$,  for a constant  $Q > 0$, it follows that 
\begin{equation}\label{eqb2}
    m_{R_1}^{-1} \lesssim \eta^{-( b_1/\sigma_1^2)} ,\hspace{0.4cm} \widehat{m}_{R_2}^{-1} \lesssim \eta^{-2b_2 Q   \sigma_2^2}.
\end{equation}
Let us assume that for $a > 0$ specified later,
\begin{equation}\label{eqassumptiona}
    \nstarn{f_A}^2 \lesssim \eta^{a}.
\end{equation}
Then,  combining \eqref{eqb0}, \eqref{eqb1} and \eqref{eqb2}, from \eqref{equationL2bound} we obtain that 
\begin{equation}
   \|f_A\|_{L^2}^2 \lesssim_{\log(\eta^{-1})}  \eta^{-\left(\frac{b_1}{\sigma_1^2} + 2b_2\sigma_2^2C\right)} \left(\eta^{a}   +  \eta^{\left(\frac{b_1}{\sigma_1^2} + \frac{b_1(1-r)^2}{\sigma_2^2}\right)} \right) + R_{2,\eta}^{-2r_2 + 1}.
\end{equation}
\begin{comment}
We note that, since $\sigma_2^2 < 2\sigma_1^2$, 
\begin{equation}
    \frac{1}{\sigma_1^2} < \frac{1 + (1-r)^2}{\sigma_2^2} 
\end{equation} 
if $r > (\sigma_2^2 - \sigma_1^2)/{\sigma_1^2}$, which is satisfied  since  $(\sigma_2^2 - \sigma_1^2)/\sigma_1^2 \in (0,1)$.  
\end{comment}
Therefore, choosing $b_1,b_2$ small enough, there exist $c>0 $ such that
\begin{equation}
     \|f_A\|_{L^2}^2 \lesssim \eta^{c} + \left(\frac{2\log(\eta^{-1})}{|\log(||f_A||_{L^2}^2)|}\right)^{-r_2 + \frac{1}{2}}
\end{equation}
holds. From  Lemma \ref{Lemmalogcontrol}, we obtain 
\begin{equation}\label{eqb3}
    \|f_A\|_{L^2}^2 \lesssim \left(\frac{2\log(\eta^{-1})}{\log\log(\eta^{-1})}\right)^{-r_2 + \frac{1}{2}}.
\end{equation}
Once $\|f_A\|_{L^2}^2$ is controlled under condition \eqref{eqassumptiona}, we can remove this assumption by splitting the expectation according to the event $\{\nstarn{f_A}^2 \leq \eta^{a}\}$ and its complement as follows
\begin{equation}
    \E \left[\|f\|_{L^2}^2\right] = \E\left[\|f\|_{L^2}^2 \left(\one_{\{\nstarn{f_A}^2 \leq \eta^{a}\} }   +   \one_{\{\nstarn{f_A}^2 > \eta^{a}\} }\right)\right].
\end{equation}
From \eqref{eqb00} and \eqref{eqb3}, we have that 
\begin{equation}
    \E\left[\|f\|_{L^2}^2 \one_{\{\nstarn{f_A}^2 \leq \eta^{a}\} } \right] \lesssim \log(\eta^{-1})^{-r_1 + \frac{1}{2}} + \left(\frac{\log(\eta^{-1})}{\log\log(\eta^{-1})}\right)^{-r_2 + \frac{1}{2}}.
\end{equation}
On the other hand, since $\|f_A\|_{L^2}^2$ is bounded, using Markov's inequality,
\begin{align}
     \E\left[\|f\|_{L^2}^2 \one_{\{\nstarn{f_A}^2 > \eta^{a}\} } \right] &\lesssim \log(\eta^{-1})^{-r_1 + \frac{1}{2}} + \E\left[ \|f_A\|_{L^2}^2  \one_{\{\nstarn{f_A}^2 > \eta^{a}\} }\right] \\
     &\lesssim \log(\eta^{-1})^{-r_1 + \frac{1}{2}} + \eta^{-a}\E\left[\nstarn{f_A}^2 \right].
\end{align}
Choosing $f = \widehat{\varphi}_N - \varphi$  and $f_A = \widehat{\varphi}_N-\varphi_A$,  
\begin{align}\label{eqlastterm}
     \E \left[\|\varphi - \widehat{\varphi}_N\|_{L^2}^2\right] &\lesssim \log(\eta^{-1})^{-r_1 + \frac{1}{2}} + \left(\frac{\log(\eta^{-1})}{\log\log(\eta^{-1})}\right)^{-r_2 + \frac{1}{2}} + \eta^{-a}\E\left[\nstarn{\varphi_A - \widehat{\varphi}_N}^2 \right].
\end{align}
Finally, using the triangle inequality combined with  Proposition \ref{PropositionInfimumApproximationError} and Corollary \ref{CorollaryRateLpMCV}, we conclude that for a sufficiently small value of $a>0$, the last term  in \eqref{eqlastterm} decays polynomially in $\eta^{}$. Finally, under the two different scenarios, we prove the convergence rate for the established regimes.
\end{proof}

\subsubsection*{Acknowledgment} 

Francisco Pina's research is funded by the PRIDE Grant “MATHCODA: Mathematical Tools for Complex Data Structures”. 

\bibliography{bibliography}

\end{document}